\documentclass[11pt,twoside,a4paper]{amsart}
\usepackage{amssymb}

\usepackage[english]{babel}
\usepackage{amsmath}
\usepackage{amsthm,amssymb,latexsym}
\usepackage[all,cmtip]{xy}
\usepackage{url,ae}
\usepackage{t1enc}
\usepackage{fancyheadings}
\usepackage{mathrsfs}
\usepackage{graphicx}
\usepackage{eufrak} 
\usepackage{geometry}

\date{\today}

\def\1{{\bf 1}}

\def\w{\wedge}

\def\dbar{\bar\partial}

\def\C{{\mathbb C}}
\def\w{{\wedge}}

\def\M{{\mathcal M}}
\def\T{{\mathcal T}}
\def\S{{\mathcal S}}

\def\F{{\mathcal F}}
\def\I{{\mathcal I}}

\def\Hom{{\rm Hom\, }}

\def\E{{\mathcal E}}

\def\O{{\mathcal O}}

\def\Lk{{\mathcal L}}
\def\U{{\mathcal U}}

\def\J{{\mathcal J}}
\def\nbh{neighborhood }

\def\be{\begin{equation}}
\def\ee{\end{equation}}

\def\Ok{\mathcal O}

\def\J{{\mathcal J}}
 \def\Homs{{\mathcal Hom\, }}
 \def\CH{{\mathcal{CH} }}

\def\N{{\mathcal N}}
\def\Na{{\mathbb N}}
\def\V{{\mathcal V}}
\def\Kers{{\mathcal Ker\,}}
\def\PM{{\mathcal PM}}
\def\pm{{pseudomeromorphic }}

\newtheorem{thm}{Theorem}[section]
\newtheorem{lma}[thm]{Lemma}

\newtheorem{prop}[thm]{Proposition}

\theoremstyle{definition}

\theoremstyle{remark}

\newtheorem{preremark}[thm]{Remark}
\newtheorem{preex}[thm]{Example}

\newenvironment{remark}{\begin{preremark}}{\qed\end{preremark}}
\newenvironment{ex}{\begin{preex}}{\qed\end{preex}}

\numberwithin{equation}{section}

\usepackage{xcolor}

\title[]{A pointwise norm on a non-reduced analytic space}

\begin{document}

\date{\today}

\author[Mats Andersson]{Mats Andersson}

\address{Department of Mathematics\\Chalmers University of Technology and the University of
Gothenburg\\S-412 96 G\"OTEBORG\\SWEDEN}

\email{matsa@chalmers.se}


\thanks{The author  was
  partially supported by the Swedish
  Research Council}

\begin{abstract}
Let $X$ be a possibly non-reduced space of pure dimension. We introduce an 
essentially intrinsic pointwise 
Hermitian norm on
smooth  $(0,*)$-forms, in particular on holomorphic functions, on $X$.
We prove that the space of
holomorphic functions is complete with respect to the natural topology induced by
this norm.
\end{abstract}

\maketitle

\section{Introduction}
Starting with papers by Pardon and Stern, \cite{PaSt1,PaSt2},  in the early 90s, 
a lot of research on the $\dbar$-equation on a reduced singular space has been conducted 
during the last decades, e.g., \cite{BS, FOV, OvVass2, Rupp2, LR2, HePo2, HePo3, AS}  
and many others. In most of them estimates for solutions are discussed.  
A pointwise, essentially unique, norm of functions and forms on 
a reduced space $X$ is obtained via a local embedding of $X$ in
a smooth manifold $\U$ and a Hermitian norm on $\U$.  

Only quite recently there has been some work about analysis on non-reduced spaces.
The celebrated Ohsawa-Takegoshi theorem, \cite{OT},  has been generalized to encompass
extensions of holomorphic functions defined on non-reduced subvarieties $X$ 
 defined by certain multiplier ideal sheaves of a manifold $Y$, 
see, e.g., \cite{CDM, De}.  In this case
the $L^2$-norm of a function (or form) $\phi$ on the subvariety is defined as a limit of $L^2$-norms of
an arbitrary extension of $\phi$ over small neighborhoods of $X$ in $Y$.   
A pointwise, but not intrinsic, norm of holomorphic functions on a non-reduced $X$ is used by Sznajdman in \cite{Sz},  
where he proved an analytically formulated local Brian\c con-Skoda-Huneke 
type theorem on a non-reduced $X$ of pure dimension.

In this paper we introduce, given a non-reduced space $X$ of pure dimension $n$,
a pointwise Hermitian norm $|\cdot |_X$ on $\Ok_X$ such that 
$|\phi|_X^2$ is a smooth function on the underlying reduced space $Z$ for any holomorphic $\phi$.
The norm is essentially canonical where $Z$ is smooth 
whereas the extension across $Z_{sing}$ possibly depends on some choices.   
A motivation is that this norm makes it possible to formulate statements
about growth of extensions of holomorphic functions defined on a non-reduced
subvariety, cf.~Remark~\ref{kontur}.

\smallskip
By  definition there is,  locally,  an embedding 
\begin{equation}\label{basic}
i\colon X\to \U \subset\C^N,
\end{equation}
where $\U\subset\C^N$ is an open subset.
This means that we have an ordinary embedding $j\colon Z\to \U$ and
 a coherent ideal
sheaf $\J$ in $\U$ with zero set $Z$ such that the structure sheaf $\Ok_X$,
the sheaf of holomorphic functions on $X$,  is isomorphic
to $\Ok_\U/\J$. Thus we have a natural surjective mapping $i^*\colon \Ok_\U\to\Ok_X$ with kernel 
$\J$.

Recall that a holomorphic differential operator $L$ in $\U$ is Noetherian with respect to $\J$
if $L\Phi=0$ on $Z$ for $\Phi$ in $\J$.  It is well-known that locally one can find a finite set $L_1,\ldots, L_m$   of Noetherian operators 
such that $L_j\Phi=0$ on $Z$ if and only if $\Phi$ is in $\J$. The analogous statement for a polynomial ideal is a
keystone in the celebrated Fundamental principle
due to Ehrenpreis and Palamodov, see, e.g., \cite{Ho}.  
Each Noetherian operator with respect to $\J$ defines an intrinsic mapping
$\Lk\colon \Ok_X\to \Ok_Z$ by  
\begin{equation}\label{pluto}
\Lk (i^*\Phi)=j^* L\Phi.
\end{equation}
We say that $\Lk$ is a {\it Noetherian operator on $X$}.
It follows that locally there are 
 Noetherian operators $\Lk_0, \ldots, \Lk_m$ on $X$ such that  
\begin{equation}\label{ide}
\Lk_j\phi=0 \  \text{in}\  \Ok_Z, \   j=1,\ldots, m,  \ \   \text{if  and  only if}\  \ \phi=0 \   \text{in} \  \Ok_X.  
\end{equation}
Given $\Lk_j$ as in \eqref{ide}, following \cite{Sz} let us consider 
\begin{equation}\label{skata}
|\phi(z)|^2=\sum_0^m|\Lk_j\phi(z)|^2.
\end{equation}
Clearly $|\phi|=0$  in an open set if and only if $\phi=0$ there so \eqref{skata}
is a norm.
However, it depends on the choice of $\Lk_j$. For instance, \eqref{ide}  still holds 
if $\Lk_j$ are multiplied by any $h$ in $\Ok_Z$ that is generically nonvanishing on $Z$. 
The set of all Noetherian operators on $X$ is a left $\Ok_Z$-module, but it is not 
locally finitely generated since each derivation along $Z$ is Noetherian.

Our main result is the construction of an intrinsic coherent
$\Ok_Z$-sheaf $\N_X$ of Noetherian operators on $X$ where $Z$ is smooth;
see Theorem~\ref{thmA} below for the definition and precise statement.
In particular, if $\Lk_0, \ldots, \Lk_m$  are local generators for $\N_X$,
then \eqref{ide} holds.  Given such a set we define our norm
$|\cdot|_X$ by \eqref{skata}.  Clearly two sets of generators
give rise to equivalent norms.

At points where $Z$ is smooth, and in addition $\Ok_X$ is Cohen-Macaulay, one can 
represent $\Ok_X$ as a free $\Ok_Z$-module in a non-canonical way: In fact, if we have local
coordinates $(z,w)$ such that $Z=\{w=0\}$,  then there are monomials $1,w^{\alpha_1},\ldots, w^{\alpha_{\nu-1}}$
such that each 
$\phi$ in $\Ok_X$ has a unique representative
\begin{equation}\label{basic2}
\hat\phi(z,w)=\hat\phi_0(z)\otimes 1+\hat\phi_1(z)\otimes w^{\alpha_1}+\ldots+\hat\phi_{\nu-1}(z)\otimes w^{\alpha_{\nu-1}},
\end{equation}
where $\hat\phi_j$ are in $\Ok_Z$.  
it turns out that $|\cdot |_X$ is the smallest norm that, up to a constant, dominates 
\begin{equation}\label{basic4}
\big(|\hat\phi_0(z)|^2+\cdots +|\hat\phi_{\nu-1}(z)|^2\big)^{1/2}
\end{equation}
for all choices of coordinates and monomial bases, see Theorem~\ref{thmA}~(iii).

\smallskip 
In Section~\ref{pontus} we prove that each point $x\in Z_{sing}$ has a \nbh $\V$ in $Z$
such that $\N_X|_{\V\cap Z_{reg}}$ 
admits a coherent extension to $\V$. Unfortunately it is not clear whether these extensions
coincide on overlaps, but by a partition of unity we can extend our norm to the entire space $X$.

\smallskip
In \cite{AL} 
sheaves $\E_X^{0,*}$  of smooth $(0,*)$-forms and a $\dbar$-operator on $X$ were recently introduced.
The Noetherian operators extend to mappings $\E^{0,q}_X\to\E_Z^{0,q}$ and so 
our norm $|\cdot|_X$ extends to 
$(0,q)$-forms on $X$.  In particular,
if $\phi$ is a smooth function on $X$, then $|\phi|^2_X$ is a smooth function on $Z$.
One can thus discuss norm estimates for possible solutions to the $\dbar$-equation on $X$.
However, in this paper  
we focus on the completeness of  
$\Ok_X$:  

\begin{thm} \label{thmB}
Assume that $\phi_j$ is a sequence of holomorphic functions on $X$ that is a Cauchy sequence on each compact subset
with respect to the uniform norm induced by  $|\cdot |_X$.   Then there is a holomorphic function $\phi$ on $X$
such that $\phi_j\to \phi$ uniformly on compact subsets of $X$.
\end{thm}

This statement  is well-known,
but non-trivial, in the reduced case,  see, e.g., \cite[Theorem 7.4.9]{Ho}.

\begin{remark}\label{kontur} 
In the recent paper \cite{AU} we find $L^p$-estimates of extensions of holomorphic
functions $\phi$ defined on a non-reduced subvariety $X$ of a strictly pseudoconvex domain $D$,
given that certain $L^p$-norms of  $|\phi|_X$ over $Z\cap D$ are finite. This generalizes results in 
\cite{Amar,Cum}, see also \cite{AM}, in the case when $X$ is reduced.
\end{remark}

The plan of the paper is as follows. In Section~\ref{basex} we discuss a simple example that illustrates
the basic idea in the definition of $\N_X$ where $Z$ is smooth. The general construction
of $\N_X$ relies on an idea due to Bj\"ork,  \cite{Bj},
to construct Noetherian operators from so-called Coleff-Herrera currents.  In Section~\ref{pottsork}
we recall the definition of such currents as well as some other basic facts that we need.  
The proof of  Theorem~\ref{thmB}  relies on some further residue theory that we recall in 
Sections~\ref{pm} and \ref{proofB}. In the latter section we also provide a proof
of Theorem~\ref{thmB}   in case $Z$ is smooth.
For the general case we need a kind of resolution of $X$ that is described in Section~\ref{reso},
and in Section~\ref{proofb} the proof of Theorem~\ref{thmB} is concluded.

\smallskip
\noindent {\bf Acknowledgement:}   I would like to thank Richard L\"ark\"ang and H\aa kan Samuelsson~Kalm
for valuable discussions related to this paper.

\section{A basic example}\label{basex}
 Let us first consider the simplest example of a non-reduced space, the space $X$ with underlying
reduced space $Z=\{w=0\}\subset \C_{z,w}^2$ defined by the ideal $(w^{m+1})$.  Each holomorphic function
$\phi$ on $X$ has a unique representative 
$$
\hat\phi(z,w)=\hat\phi_0(z)\otimes 1+\hat\phi_1(z)\otimes w+\cdots +\hat\phi_{m-1}(z)\otimes w^{m},
$$
where $\hat\phi_j$   are holomorphic on $\C_z$.  Let us therefore tentatively define the pointwise norm
\begin{equation}\label{polly0}
|\phi(z)|_0^2=|\hat\phi_0(z)|^2+|\hat\phi_1(z)|^2+ \cdots +|\hat\phi_{m}(z)|^2, \quad  z\in \C_z.
\end{equation}
To see how \eqref{polly0} behaves under a change of coordinates, first notice that 
for $k=0,\ldots,m$ we have Noetherian operators $\Lk_k\colon \Ok_X\to\Ok_Z$ defined by
$\Lk_k\phi(z)=(\partial^k  \phi/\partial w^k)(z,0)/k!$,  
where $\Phi$ is any representative in $\C^2_{z,w}$ 
of $\phi$.  In what follows we will write $\phi$ rather than $\Phi$ in this expression. 
Notice that  $\hat\phi_k(z)=\Lk_k\phi(z)$.
%
%
If we introduce new coordinates $\zeta_b,\eta_b$\   such that 
$
w=\eta_b, \ \   z=\zeta_b+b\eta_b,
$
where $b$ is a constant, then the ideal is $(\eta_b^{m+1})$, so it is
natural to consider the norm
\begin{equation}\label{bnorm}
|\phi(\zeta)|_b^2=|\Lk_{b,0}\phi(\zeta)|^2+|\Lk_{b,1}\phi(\zeta)|^2+\cdots +|\Lk_{b,m}\phi(\zeta)|^2,
\end{equation}
where $\Lk_{b,k}\phi(z)=(\partial^k  \phi/\partial \eta_b^k)(\zeta_b,0)/k!$.
%
Since  
$$
\frac{\partial}{\partial \eta_b}=\frac{\partial}{\partial w}+b\frac{\partial}{\partial z},
$$
we have that  
\begin{equation}\label{motto}
\Lk_{b, k}\phi(z)=
\sum_{j=0}^k  b^j {{k}\choose{j}} \frac{\partial^k\phi}{\partial z^j \partial w^{k-j}}(z,0).
\end{equation}

\begin{lma} If we choose any distinct $b_0, \cdots, b_m$, one of which may be $0$, then
the $\Ok_Z$-module generated by  $\Lk_{b_\ell, k}$, $\ell,k=0, \ldots, m$,
coincides with the 
$\Ok_Z$-module generated by  $\partial^k/\partial z^j \partial w^{k-j}$, $0\le j\le k\le m $.
\end{lma}

\begin{proof}
For fixed $k\le m$,  let
$$
x_j={{k}\choose{j}} \frac{\partial^k}{\partial z^j\partial w^{k-j}}\phi, \ \ 0\le j\le k.
$$ 
It follows from the well-known invertibility of a generic Vandermonde matrix that
$x_j$ can be expressed as linear combinations of $\Lk_{b_\ell, k}$,  $\ell=0, \ldots, k$. 
\end{proof}

We define  $\N_X$ to be this $\Ok_Z$-module of Noetherian operators. In view of the second set of
generators it is coordinate invariant. In follows from the proof that 
$\N_X$ is in fact generated by $\Lk_{b_\ell, k}$, $0\le \ell \le k\le m$. 
We get the pointwise norm
\begin{equation}\label{polly}
|\phi(z)|_X^2=\sum_{0\le j,k\le m}|\Lk_{b_j, k}\phi(z)|
\sim \sum_{0\le j\le k\le m}\Big| \frac{\partial^k\phi}{\partial  z^j\partial w^{k-j}}(z,0)\Big|^2.
\end{equation} 
%
%

\section{Some preliminaries}\label{pottsork}
In this section we gather definitions and known results that will be used in this paper.

\subsection{Coleff-Herrera currents}\label{sep}
Assume that $j\colon Z\to \U\subset\C^N$ is an embedding of a reduced variety $Z$ of
pure dimension $n$. 
We say that a germ of a current $\mu$ in $\U$ of bidegree $(N,N-n)$
is a Coleff-Herrera current with support on $Z$, $\mu\in\CH^Z_\U$, if it is $\dbar$-closed, 
is annihilated by $\bar\J_Z$ (i.e., $\bar h\mu=0$ for $h$ in $\J_Z$)
 and in addition has the {\it standard extension property} SEP.  The latter
condition can be defined in the following way:
Let $\chi$ be any smooth function
on the real axis that is $0$ close to the origin and $1$ in a \nbh\ of $\infty$. Then $\mu$ has the SEP
if for any holomorphic function $h$ (or tuple $h$ of holomorphic functions) 
whose zero set $Z(h)$ has positive codimension on $Z$, 
$\chi(|h|/\epsilon)\mu\to \mu$ when $\epsilon\to 0$. 
The intuitive meaning is that $\mu$ does not carry any mass on the set $Z\cap Z(h)$.  
See, e.g., \cite[Section~5]{ACH} for a discussion.

\begin{ex}[Coleff-Herrera product]\label{apa1}
If $f_1,\ldots,f_{N-n}$ are holomorphic functions in $\U$ whose common zero set is $Z$, then
the Coleff-Herrera product
\begin{equation}
\dbar\frac{1}{f}:=\dbar\frac{1}{f_{N-n}}\w\cdots\w\dbar\frac{1}{f_1}
\end{equation}
can be defined in various ways by suitable limit processes.  Its annihilator is precisely
the ideal $\J(f)=(f_1,\ldots, f_{N-n})$. Moreover, if $A$ is a holomorphic $N$-form, then
$A\w \dbar(1/f)$ is a Coleff-Herrera current.
\end{ex}

\begin{prop}\label{apa2}
If $f_j$ are as in Example~\ref{apa1}, $\mu$ is in $\CH_\U^Z$ and $\J(f)\mu=0$, then there is
(locally) a holomorphic $N$-form $A$ such that 
\begin{equation}\label{apa3}
\mu=A\w  \dbar\frac{1}{f}.
\end{equation}
\end{prop}


The statements in Example~\ref{apa1} are due to Coleff-Herrera, Dickenstein-Sessa, and Passare
in the 80's, whereas Proposition~\ref{apa2} is due to  Bj\"ork, \cite{Bj}.
Proofs and further discussions and references
can be found in \cite{Bj} and \cite[Sections 3 and 4]{ACH}.




\subsection{Embeddings of a non-reduced space}\label{kork}
Let $i\colon X\to \U\subset\C^N$ be a local embedding of a non-reduced space of pure dimension $n$
and consider the sheaf $\Homs_{\Ok_\U}(\Ok_\U/\J, \CH^Z_\U)$,
  that is, the sheaf of currents $\mu$ in $\CH^Z_\U$ such that $\J\mu=0$.  It is indeed a sheaf over $\Ok_X=\Ok_\U/\J$ so we can write
  $\Homs_{\Ok_X}(\Ok_X, \CH^Z_\U)$.   
The  duality principle,
\begin{equation}\label{dual}
\Phi\in\J \   \text{if and only if} \  \Phi\mu=0  \ \text{for all}\  \mu\in \Homs_{\Ok_\U}(\Ok_\U/\J, \CH^Z_\U),
\end{equation}
is known since long ago, see, e.g., \cite[(1.6)]{Aext}.

\smallskip
Given a point $x$ on $X$ there is a minimal number $\hat N$ such that there is a local embedding 
$j\colon X\to \U'\subset\C^{\hat N}_z$ in a \nbh of $x$. Such a minimal embedding is unique up to biholomorphisms. Moreover,
any embedding $i\colon X\to \U\subset\C^N$,  factorizes so that, in a \nbh of $x$, 
\begin{equation}\label{fact}
X\stackrel{j}{\to}\U'\stackrel{\iota}{\to}\U:=\U'\times\U''\subset\C^N,
\quad i=\iota\circ j,
\end{equation}
where $j$ is minimal, $\U''$ is an open subset of $\C^{m}_{w''}$, $m=N-\hat N$,
$\iota(z)=(z,0)$, 
and the ideal in $\U$ is
$\J=\hat\J\otimes 1+(w''_1,\ldots, w''_m)$. 
It follows from \cite[Lemma 4]{AL} that the natural mapping
\begin{equation}\label{assi}
\iota_*\colon 
\Homs_{\Ok_X}(\Ok_X, \CH^Z_{\U'}) \to \Homs_{\Ok_X}(\Ok_X, \CH^Z_\U)
\end{equation}
 is an isomorphism, naturally expressed as $\mu'\mapsto \mu=\mu'\otimes[w''=0]$.
Here $[w''=0]$ denotes the current of integration over $\{w''=0\}$. 

\begin{remark} 
The equivalence classes in \eqref{assi} can be considered as elements of an intrinsic
$\Ok_X$-sheaf ${\omega}^n_X$ of $\dbar$-closed $(n,0)$-form on $X$, introduced in
 \cite{AL}, so that $i_*\colon {\omega}^n_X\to \Homs_{\Ok_X}(\Ok_X, \CH^Z_\U)$ is an isomorphism. 
In case $X$ is reduced,  $\omega_X^n$ is the classical Barlet sheaf, \cite{Barlet}.  
\end{remark}

\subsection{$\Ok_X$ as an $\Ok_Z$-module}\label{kolsyra}
Let $x\in X$ be a point where  $Z$ is smooth and let $i\colon X\to \U$ be a local embedding
with  coordinates $(z,w)$ in $\U$ so that
$Z=\{w=0\}$.  Then there is a finite set of monomials 
$1, w^{\alpha_1}, \ldots, w^{\alpha_{\nu-1}}$   such that  each element $\phi$ in $\Ok_X=\Ok_\U/\J$ has
a representative $\hat\phi$ in $\Ok_\U$ of the form \eqref{basic2},
where $\hat\phi_j$ are in $\Ok_Z$.  
Given a minimal set of such monomials 
the representation \eqref{basic2} is unique if and only if $\Ok_X$ is Cohen-Macaulay at $x$,
see, e.g., \cite[Proposition~3.1]{AL}.  
 %
Notice however that even when $\Ok_X$ is Cohen-Macaulay 
this representation  of $\Ok_X$ as a free $\Ok_Z$-module is not
unique; it depends on the choice of coordinates $(z,w)$, cf.~Section~\ref{basex},
and in general also on the choice of minimal set of monomials $w^{\alpha_j}$. 
Also \eqref{basic4} depends on these choices. 



\subsection{Local representation of certain currents}\label{loko}
%

Consider an open subset $\U=\U'\times\U'' \subset \C^n_z\times \C_w^{N-n}$,
let $Z=\{(z,0), \ z\in \U' \}$,  and let
$\pi\colon \U\to \U'$ be the natural projection $(z,w)\mapsto z$. 
We let $p=N-n$, $dz=dz_1\w\ldots\w dz_n$,  $dw=dw_1\w\ldots\w dw_{p}$
and use the short-hand notation
$$
\dbar\frac{dw}{w^{m+\1}}=\dbar\frac{dw_1}{w_1^{m_1+1}}\w\dbar\frac{dw_2}{w_2^{m_2+1}}\w\ldots
\w\dbar\frac{dw_p}{w_p^{m_p+1}},
$$
if $m=(m_1,\ldots,m_p)\in\Na^p$ is a multiindex. 
It is well-known, and follows immediately from the one-variable case, that if $\xi(w)$ is any smooth function then
\begin{equation}\label{snok}
\dbar\frac{dw}{w^{m+\1}}.\xi= \frac{(2\pi i)^p}{m!}
\Big(\frac{\partial}{\partial w}\Big)^m\xi(0), 
\end{equation}
where
\begin{equation}\label{mas}
\Big(\frac{\partial}{\partial w}\Big)^m=\frac{\partial^{|m|}}{\partial w_1^{m_1}\cdots \partial w_p^{m_p}}.
\end{equation}

Assume that $\tau$\  is a $(N,N-n+k)$-current in $\U$ with support on 
$Z$ that is annihilated by all $\bar w_j$ and $d\bar w_\ell$.  
Then we have a (locally finite) 
unique representation
\begin{equation}\label{pelargon}
\tau=\sum_\alpha  \tau_\alpha(z)\w dz\otimes \dbar\frac{dw}{w^{\alpha+\1}},
\end{equation}
where $\tau_\alpha\w dz$   are $(n,k)$-currents on $Z$.  In fact,
\begin{equation}\label{pelargon2}
(2\pi i)^p \tau_\alpha\w  dz=\pi_*(w^\alpha\tau),
\end{equation}
cf.~\cite[(2.11)]{AL}. 
Clearly $\dbar\tau=0$ if and only if $\dbar \tau_\alpha=0$ for all $\alpha$. In particular,
$\tau$ is a Coleff-Herrera current if and only if all $\tau_\alpha$ are holomorphic functions. 
We shall only consider  $\tau$ such that $\tau_\alpha$ are smooth $(0,*)$-forms on $Z$.

\subsection{Smooth functions and $(0,q)$-forms on $X$}\label{krokodil}
The definitions and statements in this subsection are from \cite[Section 4]{AL}.
Consider a local  embedding $i\colon X\to \U\subset \C^N$ as before.  
If $\Phi$ is a smooth $(0,*)$-form in $\U$, $\Phi\in\E_\U^{0,*}$,  we say that $i^*\Phi=0$,
or $\Phi\in\Kers i^*$, 
if  
$$
\Phi\w\mu=0, \quad  \mu\in \Homs_{\Ok_X}(\Ok_X,\CH_\U^Z).
$$
In case $\Phi$ is holomorphic, this is equivalent to that $\Phi\in\J$ in view of the duality principle 
\eqref{dual}.   We let
$$
\E_X^{0,*}:=\E_\U^{0,*}/\Kers i^*.
$$
be the sheaf of smooth $(0,*)$-forms on $X$. Thus we have a well-defined surjective mapping
$i^*\colon \E_\U^{0,*}\to \E_X^{0,*}$.
By a standard argument one can check that this
definition is independent of the choice of embedding.

\smallskip
One can  verify that where $Z$ is smooth, $\Phi$ is in $\Kers i^*$ if and only if it can be written as a finite sum
of terms of the form %
\begin{equation}\label{jupiter}
\Phi_1\Psi_1+\Phi_2\bar \Psi_2+\Phi_3\w d \bar\Psi_3,
\end{equation}
where $\Phi_j$ are smooth forms,
$\Psi_1\in\J$ and $\Psi_2,\Psi_3\in\J_Z$.   

If $Z$ is smooth and we have local coordinates as in Section~\ref{kolsyra}, then for each $\phi$ in
$\E_X^{0,*}$ there is a representative
$\hat\phi$ in $\E_\U^{0,*}$ of the form \eqref{basic2} where $\hat\phi_j(z)$ are in $\E_Z^{0,*}$. Moreover,
if the set of monomials is minimal and $\Ok_X$ is Cohen-Macaulay, then the representation
is unique.

\section{The sheaf $\N_X$ where $Z$ is smooth}\label{balja}

Assume that we have an embedding
\begin{equation}\label{grus}
i\colon X\to\U\subset\C^N
\end{equation}
and that the underlying reduced space $Z$ is smooth.
As before $X$ has pure dimension $n$ and $p=N-n$.
Let $x$ be a point on $Z$ and let
$\pi\colon \U \to Z$ be a holomorphic submersion onto $Z$ in a small
\nbh $\V\subset\U$ of $x$.  By definition this means that there are local coordinates 
$(z,w)$ at $x$ such that $w_1,\ldots,w_{p}$ generates $\J_Z$   and 
$\pi(z,w)=z$.  
Our sheaf $\N_X$ of Noetherian operators is defined in the following result which is our
first main theorem.

\begin{thm} \label{thmA}
Let $X$ be a non-reduced space of pure dimension such that its underlying space 
$Z$ is smooth.  

\smallskip
\noindent (i) Let \eqref{grus}  be an embedding and let
$\pi\colon \U \to Z$ be a local holomorphic submersion onto $Z$. 
For each $\mu$ in $\Homs_{\Ok_X}(\Ok_X, \CH^Z_\U)$ and each nonvanishing
holomorphic $n$-form $dz$ there is a Noetherian operator $\Lk$ on $X$ such that
\begin{equation}\label{ara1}
dz\,  \Lk\phi=\pi_*(\phi\mu).
\end{equation}
Together these  
$\Lk$ form a coherent $\Ok_Z$-module
$\N_{X,\pi}$.  Moreover, $\Lk\phi=0$ for all $\Lk$ in $\N_{X,\pi}$ in a \nbh of a given point
if and only if $\phi=0$ there.   

\smallskip

\noindent   (ii)   If $\pi^\ell$ is a suitable finite number of generic  holomorphic submersions, then 
the associated $\Ok_Z$-modules $\N_{X,\pi^\ell}$ together generate a coherent $\Ok_Z$-module $\N_X$ such that
$\N_{X,\pi}\subset\N_X$ for any local submersion $\pi$.

\smallskip

\noindent  (iii)     Locally, where $\Ok_X$ is Cohen-Macaulay, the resulting norm $|\phi|_X$, as defined in the 
introduction,   is the smallest norm that dominates, up to constants,
any expression \eqref{basic4}.
\end{thm}

 In other words, all mappings of the form $\phi\mapsto \pi_*(\phi\mu)$ together 
 form a coherent $\Ok_Z$-sheaf $\widehat \N_X$ of $K_X$-valued Noetherian operators
  and $\N_X=\widehat\N_X\otimes K_X^{-1}$. 
 The rest of this section will be devoted to the proof of this theorem.

\subsection{Proof of Theorem~\ref{thmA}~(i)}
Assume that we have coordinates $(z,w)$ in an open set
$\V\subset\subset \U$ so that $\pi(z,w)=z$. 
Let $\mu$ be a current in  $\Homs_{\Ok_X}(\Ok_X, \CH^Z_\U)$.
By the  Nullstellensatz there is a multiindex $M$ such that $w^{M_j+1}\mu=0$ for each $j$.
Let
\begin{equation}\label{grobian}
\hat\mu=\frac{1}{(2\pi i)^p}\dbar\frac{dw}{w^{M+\1}}.
\end{equation}
In view of Proposition~\ref{apa2}, possibly after shrinking $\V$, 
there is a holomorphic function $a(z,w)$ 
such that 
\begin{equation}\label{postis}
\mu= a dz \w \hat\mu 
\end{equation}
in $\V$. 
It follows that $\Phi a dz\w \hat\mu =\Phi\mu=0$ if $\Phi$ is in $\J$ so that
$\phi a dz\w \hat\mu$ is well-defined for $\phi$ in $\Ok_X$. 
%
Let $\Phi$ be holomorphic in $\U$ such that $\phi=i^*\Phi$. Then, cf.\ \eqref{snok},  
\begin{equation}\label{skvader}
dz\, \Lk\phi=\pi_*(\phi\mu)= \pi_* (\Phi a dz\w \hat\mu )=dz
\frac{1}{M!}\big(\frac{\partial}{\partial w}\big)^M (a\Phi)\Big|_{w=0}. 
\end{equation}
Thus $\phi\mapsto \Lk\phi$ is a Noetherian operator $\Ok_X\to \Ok_Z$ on $X$
induced, cf.\  \eqref{pluto},  by the Noetherian differential operator
\begin{equation}\label{skvader99}
L\Phi=\frac{1}{M!}\big(\frac{\partial}{\partial w}\big)^M(a\Phi)
\end{equation}
in $\U$. 
If  
$\xi$ in $\Ok_Z$, then 
$$
dz \, \xi\Lk\phi=\xi\pi_*(\phi\mu)=\pi_*(\phi \pi^*\xi\mu)
$$
and since $\pi^*\xi$ is holomorphic in $\V$, 
$\pi^*\xi\mu$ is in $\Homs_{\Ok_X}(\Ok_X, \CH^Z_\U)$.  Thus $\N_{X,\pi}$ is an $\Ok_Z$-module. 

For each $\mu$ and multiindex $\gamma\le M$, define $\Lk_{\mu,\gamma}$ by
\begin{equation}\label{pucko}
dz\, \phi\mapsto dz\, \Lk_{\mu,\gamma} \phi=  \pi_*(w^{\gamma}\phi\mu). 
\end{equation}
If $\psi$ is in $\Ok_X$ and $\psi=i^*\Psi$,
where
$
\Psi=\sum_\gamma \psi_\gamma(z)w^\gamma,
$
then 
\begin{equation}\label{poker}
\Lk(\psi\phi)=\sum_\gamma \psi_\gamma\Lk_{\mu,\gamma}\phi.
\end{equation}
The $\Ok_X$-sheaf $\Homs_{\Ok_X}(\Ok_X, \CH^Z_\U)$ is coherent and thus locally 
generated by a finite number of elements, say,  $\mu_1,\ldots,\mu_\nu$.
We can assume that $M$ is chosen such that $w^{M_j+\1}\mu_k=0$ for each $j$ and $k$.
If $\mu=\psi^1\mu_1+\cdots +\psi^\nu \mu_\nu$ and $dz \Lk_k\phi=\pi_*(\phi\mu_k)$, then
\begin{equation}\label{poker2}
\Lk\phi=\Lk_1(\psi^1\phi)+\cdots + \Lk_\nu(\psi^\nu\phi).
\end{equation}
In view of \eqref{poker} and \eqref{poker2} we see that the $\Ok_Z$-module $\N_{X,\pi}$ is locally generated by
the finite set $\Lk_{\mu_k,\gamma}$, $k=1,\ldots,\nu$, $\gamma\le M$.
 %

It follows from Section~\ref{loko}, cf.~\eqref{pelargon2}, that 
$\Lk_ {\mu_k,\gamma}\phi=0$ for all $\gamma$ implies that $\phi \mu_k=0$.
By the duality principle \eqref{dual}  therefore 
$\phi=0$  if and only if $\Lk\phi=0$ for all $\Lk$ in $\N_{X,\pi}$.

Given our coordinate system $(z,w)$ in $\V$,  each $\Lk$ in 
$\N_{X,\pi}$  is uniquely expressed as
$$
\Lk=\sum_{m\le M} c_m(z) \big(\frac{\partial}{\partial w}\big)^m.
$$
It can thus be identified by the element $(c_m)$ in $\Ok^{C_M}_Z$,
where $C_M:=(M_1+1)\cdots (M_p+1)$ is the
number of multiindices $m$ such that $m\le M$. 
Moreover, this identification respects the action of  $\Ok_Z$.
Thus $\N_{X,\pi}$ is, via this identification,  a finitely generated submodule of
$\Ok^{C_M}_Z$ and therefore it is coherent.   Thus part (i) of Theorem~\ref{thmA} is proved.
%
%


\subsection{Independent submersions}\label{insub}
We will now consider a generalization of the example in Section~\ref{basex}.  
Let us assume that $\U\subset \C_z^n\times\C^p_w$ and consider the space
$X'$ with structure sheaf $\Ok_\Omega/\I$, where 
$$
\I=\langle w^{M+\1}\rangle:=\langle w_1^{M_1+1},\ldots, w_{p}^{M_{p}+1}\rangle.
$$
Notice that any local submersion $\pi$ of $\U$ onto 
$Z$ is biholomorphic to a trivial submersion $(\zeta,\eta)\mapsto \zeta$ via the change of coordinates
\begin{equation}\label{foli1}
w_k=\eta_k,\  k=1,\ldots,p, \quad  z_j=\zeta_j+\sum_{i=1}^p b_{ji}\eta_k, \  j=1,\ldots,n,
\end{equation}
where $b_{jk}$\  are holomorphic functions.
We have 
\begin{equation}\label{foli2}
\frac{\partial}{\partial\eta_k}=\frac{\partial}{\partial w_k}+\sum_{j=1}^n (b_{jk}+\Ok(w))\frac{\partial}{\partial z_j}.
\end{equation}

Let $C_m$ be the number of multiindices $\alpha=(\alpha_1,\ldots,\alpha_n)$
such that $|\alpha|\le m$.  Fix a point $0\in Z$.
Let $L$ be a set of $C_{max M_i}$ 
$n$-tuples $b^l_{\cdot }$ of holomorphic functions such that $b^l_{\cdot}(0,0)$ are generic points
in $C^n$. 
For $\ell=(\ell_1, \cdots, \ell_p)\in L^p$ consider the associated change of coordinates 
(and submersion $\pi^\ell$)  determined by 
$b^\ell_{jk}=b^{\ell_k}_j$.  Let us denote the associated derivatives $\partial/\partial \eta_k$ 
by $\partial/\partial\eta^{\ell}_k$. 
For $\gamma\le M$ we have 
the Noetherian operators
$$
\big(\frac{\partial}{\partial\eta^\ell}\big)^\gamma\psi:=
\big(\frac{\partial}{\partial\eta_p^{\ell}}\big)^{\gamma_p}\cdots \big(\frac{\partial}{\partial\eta_1^{\ell}}\big)^{\gamma_1}\psi.
$$
on $X'$. 
Here is the main result of this section.

\begin{prop}\label{fan1}
Assume that $m\le M$ and $|\beta|\le |M-m|$. 
Then there are 
holomorphic $c_{\beta,m,\ell,\gamma}$ for
$\gamma\le M$ in a \nbh of $(0,0)$ 
such that, for any $\psi$ in  $\Ok_\U/\I$,
\begin{equation}\label{spjut1}
\frac{\partial}{\partial z^\beta \partial w^m}\psi=
\sum_{\ell\in L^p} \sum_{\gamma\le M} c_{\beta,m,\ell,\gamma}\big(\frac{\partial}{\partial\eta^\ell}\big)^\gamma\psi.
\end{equation}
\end{prop}

We first consider the case when $p=1$ and the tuples $b^\ell_\cdot$ are constant.
Then we just have one variable $\eta=w$. 

\begin{lma}\label{krake}
If we choose $C_m$ generic constant $n$-tuples $b^\ell$, 
then for each $\alpha$ with $|\alpha|\le m$ there are unique $d_{\ell,\alpha}$ such that
$$
\frac{\partial^m}{\partial^\alpha z\partial^{m-|\alpha|} w}\psi=\sum_\ell d_{\ell,\alpha}\big(
\frac{\partial}{\partial\eta^\ell}\big)^m\psi.
$$
\end{lma}

\begin{proof}
In view of \eqref{foli2} we have 
$$
\big(\frac{\partial}{\partial\eta^\ell}\big)^m\psi=\big(\frac{\partial}{\partial w}+\sum_j b_j^\ell\frac{\partial}
{\partial z_j}\big)^m\psi=
\sum_{|\alpha|\le m} (b^\ell)^\alpha {{m}\choose{\alpha}} \frac{\partial^m}{\partial^\alpha z\partial^{m-|\alpha|} w}\psi,
$$
where
$$
(b^\ell)^\alpha =(b_1^{\ell})^{\alpha_1}\cdots(b_n^{\ell})^{\alpha_n}.
$$
We claim that the $C_m\times C_m$-matrix $B=(b^\ell)^\alpha$ is invertible if the $b^\ell$ are
generic.  If $n=1$ then  $B$ is the Vandermonde matrix from Section~\ref{basex}, and so the claim follows.
In the general case
one can argue as follows:   For given $x_\alpha\in\C^{C_m}$, consider the polynomial
$$
p(t)=\sum_{|\alpha|\le m}  x_\alpha t^\alpha
$$ 
in $\C^n_t$. We get the action of the matrix $B$  on $x_\alpha$ by evaluating $p(t)$ 
at the various points $b^\ell$.  Now  $B(x_\alpha)=0$ means that $p(t)$ 
vanishes at these  $C_m$ generic points, and hence $p(t)$ must vanish identically. This
means that $(x_\alpha)=0$ and since $(x_\alpha)$ is arbitrary, $B$ is invertible. 
Now the lemma follows by taking 
$$
x_\alpha={{m}\choose{\alpha}} \frac{\partial^m}{\partial^\alpha z\partial^{m-|\alpha|} w}\psi.
$$
\end{proof}

\begin{proof}[Proof of Proposition~\ref{fan1}]
Let us first assume that all $b^l_\cdot$ are constant. Then 
$\partial/\partial\eta^\ell_k$ only depends on $\ell_k\in L$, cf.~\eqref{foli2},  so 
we can denote it by $\partial/\partial\eta^{\ell_k}_k$.  Thus  
$$
\big(\frac{\partial}{\partial\eta^\ell}\big)^\gamma\psi=
\big(\frac{\partial}{\partial\eta_p^{\ell_p}}\big)^{\gamma_p}\cdots \big(\frac{\partial}{\partial\eta_1^{\ell_1}}\big)^{\gamma_1}\psi.
$$
Notice  that we can write the left hand side of \eqref{spjut1} 
(in a non-unique way) as 
\begin{equation}\label{spjut2}
\frac{\partial}{\partial z^{\beta_p}\partial w_p^{m_p}}\cdots
\frac{\partial}{\partial z^{\beta_2}\partial w_2^{m_2}}
\frac{\partial}{\partial z^{\beta_1}\partial w_1^{m_1}}\psi,
 %
\end{equation}
where $\beta_j\le M_j-m_j$. 
It follows from Lemma~\ref{krake}  that the proposition holds if $p=1$.
We now proceed by induction over $p$.
Assume that \begin{equation}\label{spjut3}
\omega:=\frac{\partial}{\partial z^{\beta_{p-1}}\partial w_{p-1}^{m_{p-1}}}\cdots
\frac{\partial}{\partial z^{\beta_2}\partial w_2^{m_2}}
\frac{\partial}{\partial z^{\beta_1}\partial w_1^{m_1}}\psi
\end{equation}
is a linear  combination of 
\begin{equation}\label{orkad}
\big(\frac{\partial}{\partial\eta_{p-1}^{\ell_{p-1}}}\big)^{\gamma_{p-1}}\cdots\big(\frac{\partial}{\partial\eta_1^{\ell_1}}\big)^{\gamma_1}\psi
\end{equation}
for $\gamma_i\le M_i$ and various $\ell_i$. 
By Lemma~\ref{krake} we have that
$$
\frac{\partial}{\partial z^{\beta_p}\partial w_p^{m_p}}\omega
$$
is a linear combination of 
$$
\big(\frac{\partial}{\partial\eta_{p}^{\ell_p}}\big)^{\gamma_{p}}\omega
$$
for $\gamma_p\le M_p$ and $(\ell_{p-1}, \ldots,\ell_1)\in L^{p-1}$. 

\smallskip

For the general case we will proceed by induction over $M$. 
First assume that $M=0$. Then \eqref{spjut1} just states that $\psi=\psi$.  
Let us  assume that we have proved the proposition for each $M'\le M$ such that
$|M'|<|M|$.  
Let $y$ be the tuple of all 
$(\partial/\partial\eta^\ell)^\gamma\psi$ for $\gamma\le M$, $\ell\in L^p$.
Let $x$ be the tuple of all 
\begin{equation}\label{ping}
\frac{\partial}{\partial z^\beta \partial w^m}\psi
\end{equation}
for $m\le M$, $|\beta|\le |M-m|$.
Then
$$
y= Bx + Ex'.
$$
where $Bx$ is the terms obtained when expanding $(\partial/\partial\eta^\ell)^\gamma\psi$ by \eqref{foli2}
and no derivative. falls on any $b^\ell$ or any $\Ok(w)$.
Each entry  in $x'$ must then be of the form  \eqref{ping} where 
$m<M$ and $|\beta|<|M-m|$. It follows from the induction hypothesis
that $x'$ is in the module spanned by $y$.
Therefore $Ex'=Dy$ for some holomorphic matrix $D$ in a \nbh of $(0,0)$
and hence
$
Bx=Dy-y.
$
Notice that $B(0,0)$ is injective in view of the constant case of the proposition. Hence it is
injective and has a holomorphic left inverse $A$  in a \nbh of $(0,0)$ and so 
$
x=A(Dy-y).
$
\end{proof}




\subsection{Proof of Theorem~\ref{thmA}~(ii)}
Let us assume that $\pi$ is defined by local coordinates $(z,w)$ and for $\ell\in L^p$ let
$(\zeta^\ell,\eta)$ be the local coordinates in Proposition~\ref{fan1} corresponding to 
$\partial/\partial\eta^\ell$ and  let $\pi^\ell$ be the induced local holomorphic submersion. 
 Notice that 
$
dw\w dz=B_\ell d\eta\w d\zeta^\ell,
$
where $B_\ell=1+\Ok(\eta)$, hence invertible in $\Ok_{X'}$, so that
\begin{equation}
d\eta\w d\zeta^\ell=A_\ell dw\w dz,
\end{equation}
where $A_\ell$ is the inverse of $B_\ell$. 
Notice that, cf.~the argument for \eqref{skvader} above and \eqref{grobian},
\begin{multline}\label{fan3}
dz\big(\frac{\partial}{\partial\eta^\ell}\big)^m\psi=d\zeta^\ell\big(\frac{\partial}{\partial\eta^\ell}\big)^m\psi
=c_m \pi^\ell_*(\psi \eta^{M-m}\hat\mu \w d\zeta^\ell)=\\
c_m\pi^\ell_*(\psi \eta^{M-m}A_\ell\hat\mu \w d\zeta)=
\sum_{\gamma\le m}A_{\ell,\gamma}(z) \pi^\ell_*(\psi \eta^{M-\gamma}\hat\mu\w d\zeta).
\end{multline}

Given $\mu$ in $\Homs(\Ok_X,\CH^Z_\U)$ there is a holomorphic $a$, cf.~\eqref{postis}, such that 
 $\mu=a\hat\mu\w d\zeta$.  
 From \eqref{skvader} we have
\begin{equation}\label{skvader2}
 \pi_*\Big(\phi a \hat\mu \w dz\Big)=dz
\frac{1}{M!}\big(\frac{\partial}{\partial w}\big)^M (a\phi).
\end{equation}
From Proposition~\ref{fan1} we have that
\begin{equation}\label{skvader3}
\frac{1}{M!}\big(\frac{\partial}{\partial w}\big)^M (a\phi)=
\sum_{\ell\in L^p}\sum_{m\le M} c_{\ell,m}\big(\frac{\partial}{\partial \eta^\ell}\big)(a\phi).
\end{equation}
Combining \eqref{fan3} with $\psi=a\phi$, \eqref{skvader2} and \eqref{skvader3} we get 
\begin{equation}\label{fan4}
\pi_*(\phi\mu)=\sum_{\ell\in L^p}\sum_{m\le M} c'_{m,\ell}(z)\pi^\ell_*(\phi \eta^{M-m}\mu).
\end{equation}
In view of the last part of the proof of Theorem~\ref{thmA}~(i) we thus have:

\begin{lma} \label{ponton}
Let $\mu_k$ be a finite set of generators for  $\mu\in\Homs(\Ok_X,\CH^Z_\U)$.
If $\Lk$ is defined by  $dz\, \Lk\phi=\pi_*(\phi \mu)$,  then $\Lk$ belongs to the $\Ok_Z$-module
generated by $\Lk_{\mu_k,\gamma,\ell}$, where
$$
dz\, \Lk_{\mu_k,\gamma,\ell}\phi=\pi^\ell_*(w^\gamma \phi \mu_k).
$$
\end{lma}

We conclude that, given an embedding, $\N_X$ defined so far, is 
independent of coordinates. 
It remains to see the independence of the embedding. 
Let us choose coordinates $(z,w)$ in our original
embedding so that
our embedding factorizes over a minimal embedding as in Section~\ref{kork} and write
$w=(w',w'')$. Then each 
$\mu$ in $\Homs(\Ok_X,\CH^Z)$ has the form
$\hat\mu\otimes[w''=0]$. Thus it can be written
$$
\mu=a(\zeta,w')\frac{1}{(2\pi i)^p}\dbar\frac{dw'}{(w')^{M+\1}}\w\dbar\frac{dw''}{w''}\w dz.
$$
It follows that if the submersion $\pi$ is represented by \eqref{foli1},
then the fiber over $z=\zeta$ is parametrized by $\eta\mapsto (z+b'\eta'+b''\eta''; \eta',\eta'')$
so that 
\begin{multline*}
\pi_*(\phi\mu)(z)= dz \int_\eta a(z,\eta')\phi(z+b'\eta'+b''\eta''; \eta',\eta'')\frac{1}{(2\pi i)^p}\dbar\frac{dw'}{(w')^{M+\1}}\w\dbar\frac{dw''}{w''}\w dz=
\\
dz \int_{\eta'}a(z,\eta')\phi(z+b'\eta'; \eta',0)\frac{1}{(2\pi i)^{p'}}\dbar\frac{dw'}{(w')^{M+\1}}
\end{multline*}
and so we get the same result with $b'_{jk}(\zeta,\eta',0)$.
Thus $\Lk$, defined by $dz\Lk\phi=\pi_*(\phi\mu)$, occurs already from a minimal embedding
and so the proof of Theorem~\ref{thmA}~(ii) is complete. 

%

\subsection{Generators for $\N_X$ and the norm $|\cdot|_X$}

From Lemma~\ref{ponton} we get:

\begin{prop}\label{pollinera}
If $\mu_1,\ldots,\mu_\nu$ generate  $\Homs(\Ok_X,\CH_\U^Z)$ in $\U$,
$\eta_1,\ldots,\eta_p$ generate $\J_Z$ in $\U$, and $\pi^\ell$ is a suitable finite set of
submersions, then 
the operators $\phi\mapsto \Lk_{\mu_k,\gamma,\ell}$,  defined as in Lemma~\ref{ponton},
generate $\N_X$ in $\U$.
\end{prop}

By the Nullstellensatz this is a finite set since $\eta^\gamma\mu_k=0$ if $\gamma$ is large.
Here $dz$ stands for any holomorphic $n$-form on $Z$.
We will now describe another set of generators in terms of the $\mu_k$
and coordinates $(z,w)$ in $\U$.

\begin{prop}\label{stud}
Assume that  $\I=\langle w^{M+\1}\rangle$ is contained in $\J$. 
Let $a_k$ be holomorphic functions such that  
$\mu_k=a_k\hat\mu\w dz,  \quad  k=1,\ldots,\nu,$
cf.~\eqref{grobian}, in $\U$.  Then 
the  operators 
$$
\phi\mapsto \Lk_{m,\beta,k}\phi:=\frac{\partial(\phi a_k)}{\partial z^\beta\partial w^{m}}(\cdot,0), 
\quad m\le M, \ |\beta|\le |M-m|,
$$
are independent of the choice of $a_k$ and generate
the $\Ok_Z$-module $\N_X$  in $\U$.
\end{prop}

\begin{proof}
If $a_k\hat\mu=a'_k\hat\mu$, then $a_k-a_k'$ is in $\I$ by the duality principle for a complete intersection,
and hence  $\Lk_{m,\beta,k}\phi$ are independent of $a_k$ (and of the choice of representative of the class $\phi$
in $\Ok/\J$).   It follows from  Proposition~\ref{fan1} that $\Lk_{m,\beta,k}\phi$ are $\Ok_Z$-linear
combinations of $(\partial/\partial^\ell)^m(\phi a_k)$, and in view of \eqref{fan3}  
therefore of $\pi_*^\ell(\phi \mu_k)$. Hence $\Lk_{m,\beta,k}$ are in $\N_X$ by definition.
By \eqref{skvader2}, applied to $\pi_*^\ell(\phi\eta^{\gamma}a_k\hat\mu)$, we see that 
each $\Lk_{\mu_k,\gamma,\ell}$ is an $\Ok_Z$-linear combination of the $\Lk_{m,\beta,k}\phi$, and hence
$\Lk_{m,\beta,k}$ generate $\N_X$.
\end{proof}

 
We thus have
\begin{equation}\label{atlas4}
|\phi(z)|_X\sim \sum_{k=1}^\nu \sum_{m\le M}  \sum_{|\beta|\le |M|-|m|}
\big|\frac{\partial(\phi a_k)}{\partial z^\beta\partial w^{m}}(z,0)\big|.
\end{equation}
%
It follows from \eqref{atlas4} that 
\begin{equation}\label{fisk}
|\xi\phi|_X\le C|\phi|_X,
\end{equation}
where $C$ only depends on $\xi\in\Ok_X$. Notice that 
if in addition $\xi$ is invertible in $\Ok_X$,  then  
$
|\phi|_X\sim |\xi\phi|_X
$
since 
$|\phi|_X=|\xi^{-1}\xi\phi|_X\lesssim |\xi\phi|_X$.

\subsection{Proof of Theorem~\ref{thmA}~(iii)}\label{polyp}

\begin{prop} \label{kokos}
Let $i\colon X\to \U$ be a local embedding $i\colon X\to \U$ and assume that $\Ok_X$ is Cohen-Macaulay. 
Let  $\mu$ be a column of Coleff-Herrera currents that generate $\Homs_{\Ok_X}(\Ok_X, \CH^Z_\U)$,
$(z,w)$ are coordinates in $\U$,  $1, \ldots, w^{\alpha_{\nu-1}}$ a basis for $\Ok_X$
over $\Ok_Z$.  If  $\phi$ in $\Ok_X$ is represented by $\hat\phi$ as in \eqref{basic2} and
\begin{equation}\label{kolibri0}
\phi\mu=\sum_\gamma b_\gamma dz\w\dbar\frac{dw}{w^{\gamma+1}},
\end{equation}
then
\begin{equation}\label{brutus}
\sum_j |\hat\phi_j(z)|\sim \sum_\gamma |b_\gamma |.
\end{equation}
\end{prop}

\begin{proof}
Notice first that  
$
\phi\mu=\sum_j \hat\phi_j(z) w^{\alpha_j}\mu.
$
Since
\begin{equation}\label{kolibri}
\mu=\sum_\gamma \mu_\gamma dz\w \dbar\frac{dw}{w^{\gamma+1}},
\end{equation}
cf.\ \eqref{pelargon},
it follows that each $b_\gamma$ is a holomorphic linear combination of the $\hat\phi_j$. That is, 
there  is a holomorphic matrix $T$ that maps the tuple $(\hat\phi_j)$ to the tuple
$(b_\gamma)$. Thus the right hand side of \eqref{brutus} is dominated by the left hand side. 
Since $\Ok_X$ is Cohen-Macaulay, see
\cite[Lemma~4.11]{AL},  there is a holomorphic matrix $S$ that takes $(b_\gamma)$ to $(\phi_j)$, 
and therefore the reverse inequality in \eqref{brutus} holds.  
\end{proof}

It follows from \eqref{pelargon2} that 
$
\Lk_{\mu,\gamma}\phi:=\pi_*(\phi w^\gamma\mu)= (2\pi i)^p b_\gamma.
$
Since $\Lk_{\mu,\gamma}$ are in $\N_X$ it follows that \eqref{brutus} is dominated by $|\phi|_X$. 
In view of Proposition~\ref{pollinera} a suitable finite sum of  expressions \eqref{brutus}, corresponding
to various submersions, will be $\sim |\phi|_X$. 
Thus part (iii) of Theorem~\ref{thmA} is proved. 

\subsection{Norm of smooth $(0,*)$-forms}
Recall from Section~\ref{krokodil} that $\phi\w\mu$ is well-defined for 
$\phi$ in $\E_X^{0,*}$ and $\mu$ in $\Homs(\Ok_X,\CH_\U^Z)$ and that
$\phi\w \mu=0$ for all such $\mu$, by definition, if and only if $\phi=0$.
Let  $(z,w)$ be local coordinates defining the submersion $\pi$. Recall that $\phi$ has
a representative $\hat\phi$ of the form \eqref{basic2} where $\hat\phi_j$ are smooth 
$(0,*)$-forms on $Z$.  
If we express $\mu$ in $\Homs(\Ok_X,\CH_\U^Z)$ on the form \eqref{kolibri} we see that  
$$
\phi\w \mu=\sum_\alpha b_\alpha(z)\w dz \otimes \dbar \frac{dw}{w^{\alpha+\1}},
$$
where $b_\alpha$ are smooth $(0,*)$-forms on $Z$.    
Thus the associated operator  $\Lk$ in $\N_X$, defined by 
\begin{equation}\label{stolle}
dz\, \Lk\phi=\pi_*(\phi\w \mu).
\end{equation}
for holomorphic $\phi$, extends to an operator 
 $\E_X^{0,*}\to\E_Z^{0,*}$. 
 
\smallskip
The natural action of holomorphic differential operator $T$  in $\U$ on smooth functions
in $\U$ extends to an action on smooth $(0,*)$-forms. Formally by 
interpreting the derivatives as Lie derivatives;  in practice it just means that 
if 
$
\Phi=\sum'_{I,J} \Phi_{I,J}\w d\bar z_I\w d\bar w_J
$
in $\U$,  then
$
T\Phi=\sum'_{I,J} T\Phi_{I,J}\w d\bar z_I\w d\bar w_J.
$
If $T$ is Noetherian with respect to $\J$, then its pullback $\T\phi(z)$ to $Z$ 
is independent of the representative $\Phi$ of $\phi$ in $\E_X^{0,*}$.  
In fact,   if $\Phi$ is in $\Kers i^*$ then  $T\Phi=0$ on $Z$, cf.~\eqref{jupiter}.

It follows from  \eqref{skvader} and \eqref{skvader99} that if $\Lk$ is in $\N_X$ and $\phi$ is in
$\E_X^{0,*}$, then $\Lk\phi$  is induced by a Noetherian operator in $\U$. 
Also other statements above about holomorphic
$\phi$ extend to $\phi$ in $\E_X^{0,*}$ by the same proofs. 
In particular, if 
 $\Lk_j$ is a generating set for the $\Ok_Z$-module $\N_X$  we get the pointwise norm
\begin{equation}\label{potens}
|\phi|_X^2=\sum_j|\Lk_j\phi|_Z^2,
\end{equation}
where $|\Lk_j\phi|_Z$ is a norm of forms on $Z$ induced by some Hermitian metric in ambient space.
If $\Ok_X$ is Cohen-Macaulay then \eqref{potens} is the smallest norm that (up to constants) dominates
each expression \eqref{basic4} obtained from the representation  \eqref{basic2} of $\phi$.

\section{An example}\label{kokong}

Consider the $2$-plane $Z=\{w_1=w_2=0\}$ in $\U\subset\C^4_{z_1,z_2,w_1,w_2}$,
where $\U$ the product of balls $\{|z|<1, \ |w|< 1\}$ in $\C^4$,  and let
$$
\J=\langle w_1^2,w_2^2, w_1w_2, w_1z_2-w_2z_1\rangle.
$$
Then $\Ok/\J$ has pure dimension $2$ and is Cohen-Macaulay except at the point $0\in\U$,
see, \cite[Example~6.9]{AL}. It is also shown there that $\Homs_\Ok(\Ok/\J,\CH^Z_\U)$ is generated by
$$
\mu_1=\dbar\frac{dw_1}{w_1}\w\dbar\frac{dw_2}{w_2}\w dz_1\w dz_2,\quad
\mu_2=(z_1w_2+z_2w_1)\dbar\frac{dw_1}{w_1^2}\w\dbar\frac{dw_2}{w_2^2}\w dz_1\w dz_2.
$$
Following the recipe in Proposition~\ref{stud}, $\mu_1$ only gives rise to the Noetherian operator $1$.
However, $\mu_2$ gives rise to all Noetherian operators obtained by the action of
$$
\frac{\partial}{\partial w_1},  \frac{\partial}{\partial w_2}, \frac{\partial}{\partial z_1}, \frac{\partial}{\partial z_2},
\frac{\partial^2}{\partial z_1\partial w_1}, \frac{\partial^2}{\partial z_1\partial w_2},
\frac{\partial^2}{\partial z_2\partial w_1}, \frac{\partial^2}{\partial z_2\partial w_2},
\frac{\partial^2}{\partial w_1\partial w_2}
$$
on $(z_1w_2+z_2w_1)\phi$.  They are
\begin{equation}\label{struma}
z_1\phi, z_2\phi, 0, 0, z_2\frac{\partial}{\partial z_1}, (1+z_1\frac{\partial}{\partial z_1})\phi,
(1+z_2\frac{\partial}{\partial z_2})\phi, z_1\frac{\partial}{\partial z_2}\phi, 
(z_1\frac{\partial}{\partial w_1}+z_2\frac{\partial}{\partial w_2})\phi.
\end{equation}
In view of the presens of $1$ from the $\mu_1$, we can forget about $z_j\phi$ and replace 
$1+z_j\frac{\partial}{\partial z_j}$ by $z_j\frac{\partial}{\partial z_j}$. Thus we get 
that
$$
|\phi|_X^2\sim |\phi|^2 +|z|^2\big|\frac{\partial\phi}{\partial z_1}\big|^2+|z|^2\big|\frac{\partial\phi}{\partial z_2}\big|^2
+\big|z_1\frac{\partial\phi}{\partial w_1}+z_2\frac{\partial\phi}{\partial w_2}\big|^2.
$$

\subsection{Functions in $X\setminus\{0\}$}\label{opium}
Let $\Lk_0=1$ and let $\Lk$ denote the right-most operator in \eqref{struma}.  
If $\phi$ is a $\Ok_X$-function defined in $Z\setminus\{0\}$, then both
$\Lk_0\phi$ and $\Lk\phi$ are holomorphic functions in $Z\setminus\{0\}$. Thus they
both have holomorphic extensions across $0$ that we denote by $\phi_0(z)$ and $h(z)$,  respectively.

Notice that $1, w_1$ is a basis for $\Ok_X$ over $\Ok_Z$ where $z_1\neq 0$, and similarly,
$1, w_2$ is a basis for $\Ok_X$ over $\Ok_Z$ where $z_2\neq 0$.
Notice that given any $\phi_0$ and $h$ in $\U$ we get  a $\Ok_X$-function $\phi$ in
$\U\setminus\{0\}$, defined as
\begin{equation}\label{struma2}
\phi=\phi_0+(h/z_1)w_1, \ z_1\neq 0;  \quad \phi=\phi_0+(h/z_2)w_2, \  z_2\neq 0.
\end{equation}
It is readily checked that $\Lk_0\phi=\phi_0$ and $\Lk\phi=h$. In other words, there is a $1-1$ correspondence
between  $\O_X$-functions $\phi$
in $Z\setminus\{0\}$  and $\Ok_Z^2$.

\begin{lma}\label{opium2}
The $\Ok_X$-function $\phi$ has an extension across $0$ if and only if $h(0)=0$.
\end{lma}

\begin{proof}
If $\phi$ is defined in $\U$ then $h=\Lk\phi$ in $\U$ and then clearly $h(0)=0$.  Conversely, if $h(0)=0$, then
$h(z)=c_1(z)z_1+c_2(z)z_2$ for some functions $c_1,c_2$ in $\U$. It is  readily checked that indeed
$\phi$, defined by \eqref{struma2},  coincides with 
$$
\phi_0(z)+c_1(z)w_1+c_2(z)w_2
$$
in $\U\setminus\{0\}$. Thus $\phi$ extends across $0$.
\end{proof}

In view of this lemma, if we take, e.g.,  $h=1$ in \eqref{struma2}, we get an $\Ok_X$-function $\phi$ in
$\U\setminus\{0\}$ that does not extend across $0$.

\section{Extension of $\N_X$ across $Z_{sing}$}\label{pontus}

Let $x$ be a fixed point on $Z_{sing}$.

\begin{lma} \label{raket}
There is a \nbh $\U$ of $x$
and  complete intersection  $Z_f=\{f_1=\cdots =f_p=0\}$ in $\U$ that contains $Z\cap\U$ 
and  such that $df:=df_1\w\ldots\w df_p$ is nonvanishing on $Z_{reg}\setminus Z(f)_{sing}$.
If $x'\in\U\setminus Z$ we can assume that $x'\notin Z(f)$.
\end{lma}

That is,  $Z(f)$ may have "unnecessary" irreducible components, but $df\neq 0$ at each point on 
$Z_{reg}$ that is not hit by any of these components.   
This is a well-known result but we have found no
good reference so we provide an argument. 

\begin{proof}  In a small enough \nbh  $\U$ of $x$ we can find a finite number of functions $g_1,\ldots, g_m$
that generate $\J_Z$. 
For each irreducible component $Z^\ell$ of $Z$ we choose a point $x^\ell\in Z^\ell_{reg}\cap \U$. 
Notice that $dg_j$ span the annihilator of the tangent bundle at $x^\ell$ for each $\ell$. 
If $f_1,\cdots, f_p$ are generic linear combinations of the  $g_j$, then
$df_j$ span these spaces as well for each $\ell$,  and  $f_j$ define a complete intersection $Z(f)$
that avoids $x'$.
Clearly $Z\subset Z(f)$ and $df\neq 0$ at $x^\ell$ for each $\ell$.  
It follows from \cite[Theorem 4.3.6]{JoPf} 
that $df$ is nonvanishing on the regular part of the irreducible component of $Z(f)$ that
contain $x^\ell$; i.e., on  $Z^\ell_{reg} \setminus Z(f)_{sing}$,  for each $\ell$.
\end{proof}




Let $f$  and $\U$ be as in the lemma and let us write $Z$ rather than $Z\cap\U$ etc. 
Since $df$ is generically nonvanishing on $Z$ we can choose
 coordinates 
$
 (\zeta,\eta)=(\zeta_1,\cdots,\zeta_n;\eta_1\cdots,\eta_p)
$ 
 in $\U$ such that 
$
H=\partial f/\partial\eta
$
is generically invertible on $Z$.  Let $h=\det H$. 
If 
\begin{equation}\label{mars}
w=f(\zeta,\eta), \  z=\zeta,
\end{equation}
then
$
dw\w dz=h d\eta\w d\zeta
$
and hence $(z,w)$ are local coordinates at each point on $Z\setminus\{h=0\}$.
Notice that
  \begin{equation}\label{svin1}
\frac{\partial}{\partial w}=H^{-1} \frac{\partial}{\partial \eta}, \
\quad \frac{\partial}{\partial z}=\frac{\partial}{\partial \zeta}-G\frac{\partial}{\partial w},
\end{equation}
where $G=\partial f/\partial\zeta$ is holomorphic. 
Since $H=\Theta/h$, where $\Theta$ is holomorphic,  
therefore 
\begin{equation}\label{svin2}
h\frac{\partial}{\partial w}=\Theta \frac{\partial}{\partial \eta}, \quad \
h\frac{\partial}{\partial z}=h\frac{\partial}{\partial \zeta}-G \Theta\frac{\partial}{\partial \eta}.
\end{equation}
For a sufficiently large multiindex $M=(M_1,\ldots,M_p)$ the complete intersection ideal
$\langle f^{M+\1}\rangle:=\langle f_1^{M_1+1},\ldots, f_p^{M_p+1}\rangle$ is contained in $\J$.  
Possibly after shrinking the \nbh $\U$ of $x$ there are generators $\mu_1,\ldots, \mu_\nu$ for $\Homs(\Ok_X,\CH_\U^Z)$ 
and holomorphic functions $a_k$ in  $\U$, cf.~Proposition~\ref{apa2},  such that
\begin{equation}\label{knorr}
\mu_k=a_k \frac{1}{(2\pi i)^p}\dbar\frac{1}{f^{M+1}}\w d\zeta\w d\eta.
\end{equation}
Notice that $a_k$ must vanish on the "unnecessary" irreducible components of $Z(f)$.

\begin{prop}\label{snabel2} 
With the notation above, the differential operators
\begin{equation}\label{gris}
\phi\mapsto \Lk_{m,\beta,k}\phi:= \Big(h\frac{\partial}{\partial w}\Big)^m \Big(h\frac{\partial}{\partial z}\Big)^\beta(a_k\phi),
\quad m\le M, \   |\beta|\le |M-m|, \ k=1,\ldots, \nu, 
\end{equation}
a~priori defined on $Z_{reg}\cap\{h\neq 0\}$, have holomorphic extensions to $Z$.
Moreover, they  belong to $\N_X$ on $Z_{reg}$ and generate the $\Ok_Z$-module
$\N_X$ where $h\neq 0$.
\end{prop}

\begin{proof} 
It follows immediately from \eqref{svin2} that 
$\Lk_{m,\beta,k}$ have  holomorphic extensions to $Z$.   Since $(z,w)$ are local coordinates at a point on $Z_{reg}$ where
$h\neq 0$ it follows from Proposition~\ref{stud} that $\Lk_{m,\beta,k}$ belong to $\N_X$.  
By a simple induction argument it follows from the same proposition that they actually generate 
$\N_X$ there.  
We have to prove that $\Lk_{m,\beta,k}$ are in $\N_X$ on $Z_{reg}$ where $h=0$. 
Let $x'\in X_{reg}$ be such a point and assume that $df\neq 0$. 
For a generic choice of constant matrices $b,c$ we have that
$df\w d(c\zeta+b\eta)\neq 0$. Thus 
we can choose new coordinates 
$$
w'=f(\zeta,\eta), \quad z'=c\zeta+b\eta
$$
in a \nbh $\V$ of $x'$.  It follows that  %
\begin{equation}\label{skor}
\phi\mapsto\frac{\partial}{\partial (w')^m \partial (z')^\alpha} (\phi a_k), \quad m\le M, \ |\alpha|\le|M-m|,
\end{equation}
are in $\N_X$ in  $\V$.  
Since   $z'=bz+cw $, $w'=w$, in $\V\setminus\{h=0\}$, we have
$$
\frac{\partial}{\partial w}=\frac{\partial}{\partial w'}+\frac{\partial z'}{\partial w}\frac{\partial}{\partial z'}, \quad
\frac{\partial}{\partial z}=\frac{\partial z'}{\partial z}\frac{\partial}{\partial z'},
$$
so  from \eqref{svin2} we get that 
\begin{equation}\label{skor2}
h\frac{\partial}{\partial w_j}=h\frac{\partial}{\partial w_j'}+\sum_k d_{jk} \frac{\partial}{\partial z'_j},
\quad h\frac{\partial}{\partial z}=\sum_k d'_{jk} \frac{\partial}{\partial z'_j},
\end{equation}
where $d_{jk}, d'_{jk}$ %
have holomorphic extensions to $\V$.   
Thus  $\Lk_{m,\beta,k}$ are $\Ok_Z$-linear
combinations of \eqref{skor}, and hence in $\N_X$.

\smallskip
From now on we consider a point $x'\in Z_{reg}$ where $df=0$, i.e., some "unneccessary"
component of $Z(f)$ passes through $x'$.  Then certainly $h(x')=0$.  
Let $\pi$ be the projection $(z,w)\mapsto z$. By \eqref{mars}, \eqref{knorr}, and
\eqref{skvader}, 
\begin{equation}\label{plutt1}
\frac{1}{\gamma!}dz \, \big(\frac{\partial}{\partial w}\big)^\gamma (a_k\phi)=\pi_*\big(\phi a_k df\w dz\w f^{M-\gamma}\frac{1}{(2\pi i)^p}\dbar\frac{1}{f^{M+\1}}\big), \quad \gamma\le M, 
\end{equation}
in $\V\setminus\{h=0\}$.
Let $v=(v_1,\ldots, v_p)$ generate $\J_Z$ at $x'$
and assume first that $dv\w dz\neq 0$ so that $(z,v))$ are local coordinates in a
\nbh $\V$ 
of $x'$.   
Since $\J(f)\subset\J_Z$,  $f=Av$ for a holomorphic matrix $A$ in $\V$. At points $z\in Z\cap\V\setminus\{h=0\}$ 
both $f_j$ and $v_j$ are minimal sets of generators for $\J_Z$ so $A$ is invertible there.
Therefore also $(z,v)$ define the submersion $\pi$ in $\V\setminus\{h=0\}$.
 Since $f^{M-\gamma} df\w dz=\alpha d\eta\w d\zeta$, where $\alpha$ is holomorphic,
the right hand side of \eqref{plutt1} is 
$\pi_*(\phi \alpha \mu_k)$
which is $d\zeta$ times an element in $\N_X$ in $\V$. It follows that the left hand side of \eqref{plutt1} is
$d\zeta$ times $\Lk\phi$, where $\Lk$ extends to an element in $\N_X$ in $\V$.

Notice that if $b^\ell$ is a small constant $n\times p$-matrix, then $\eta'=\eta, \zeta'=\zeta+b^\ell f$
is a change of variables in $\V$, possibly after shrinking our  \nbh $\V$ of $x'$.  In fact,
$
d\eta'\w d\zeta'=d\eta\w(d\zeta+b^\ell df)
$
is nonvanishing if $b^\ell$ is small enough. Taking $w^\ell=f$, $z^\ell=\zeta'$, we get that
$w^\ell=w$, $z^\ell= z+b^\ell w$,  and hence
$$
dw^\ell\w dz^\ell= df\w d(\zeta+b^\ell df)=df\w d\zeta=h d\eta\w d\zeta, 
$$
where $h$ is the same function as in \eqref{skor2}.  As in the the preceding step of the proof
we conclude that
$$
\phi\mapsto \big(\frac{\partial}{\partial w^\ell}\big)^\gamma(a_k\phi), \quad \gamma\le M, 
$$
a~priori defined in $\V\setminus\{h=0\}$, have extensions to elements in $\N_Z$ in $\V$. 
By Proposition~\ref{fan1} there are  a finite set of such $b^\ell$ and holomorphic 
$d_{m,\beta,\ell,k}$ in a possible even smaller \nbh $\V$ of $x'$ such that
\begin{equation}\label{patrik}
 \big(\frac{\partial}{\partial w}\big)^m\big(\frac{\partial}{\partial z}\big)^\beta(a_k\phi)=
 \sum_\ell \sum_{\gamma\le M}  d_{m,\beta,\ell,\gamma}\big(\frac{\partial}{\partial w^\ell}\big)^\gamma(a_k\phi),
 \quad m\le M, \    |\beta|\le|M-m|.
\end{equation}
It  follows  
that all the operators on the left hand side of \eqref{patrik} 
are in $\N_X$ in $\V$.

Finally, if $dv\w d\zeta=0$ at $x'$ we introduce new coordinates $z'=c\zeta+b\eta, w'=w$ as
before so that $dv\w dz'\neq 0$. From what we have just proved, then
all 
$$
\big(\frac{\partial}{\partial w'}\big)^m\big(\frac{\partial}{\partial z'}\big)^\beta(a_k\phi)
$$
are holomorphic at $x'$.  It now follows from
\eqref{skor2} that  $\Lk_{m,\beta,k}$ are in $\N_X$ at $x'$. Thus Proposition~\ref{snabel2} is proved.     
\end{proof}

We can now formulate our main result of this section.

\begin{thm}\label{oxet}
Assume that $i\colon X\to \U\subset\C^N$ is an embedding.
Given a point $x\in Z_{sing}$ there is a \nbh $\U'\subset \U$ and a 
finite number of Noetherian differential operators  $\Lk_1, \ldots, \Lk_\nu$ in $\U'$ that
generate $\N_X$  on $\U'\cap Z_{reg}$. 
\end{thm}

Clearly, such a set of $\Lk_k$ define a coherent extension of $\N_X$ to $\U'\cap Z$.

\begin{proof}
Let $(\zeta,\eta)$ be global coordinates $(\zeta,\eta)$ in $\U$.  Proposition~\ref{snabel2} provides a 
\nbh $U'$ of $x$ and set of
Noetherian differential operators in $\U'$ that belong to $\N_X$ in $\U\cap Z_{reg}$ and
generate $\N_X$ on $\U\cap (Z_{reg}\setminus \{h=0\}$,
where $h=\det(\partial f/\partial\eta)$.   Applying the same proposition to a finite number of 
other global coordinate systems (or more simply, letting $\eta$ be other choices of
$p$ special coordinates among the given ones) we get a finite number of Noetherian differential operators 
in a possibly smaller \nbh $\U'$ of $x$ that belong to $\N_X$ and generate $\N_X$ $\U'\cap (Z_{reg}\setminus \{df=0\}$.
That is, our differential operators defined so far generate $\N_X$ everywhere on $Z_{reg}\cap \U'$
except on $Z_{reg}\cap Z(f)_{sing}$.  

If we make the same construction for a finite number $f^1,\ldots, f^\nu$
of complete intersections such that 
$Z=Z(f^1)\cap\cdots \cap Z(f^\nu)$, see Lemma~\ref{raket},
we thus get a \nbh $\U'$ of $x$ and finite set of differential operators as desired.
\end{proof}

\begin{ex}\label{kollekt}
Let $Z=\{f=0\}$ be a reduced subvariety of 
$\U\subset \C^{n+1}_\zeta$ and that $df\neq 0$ on $Z_{reg}$.  If $X$ is defined by 
$\J=\langle f^2\rangle $, then
$$
\mu = \dbar\frac{1}{f^2}\w d\eta \w d\zeta
$$
is a generator for $\Homs(\Ok_X, \CH_\U^Z)$.  
If we give $\zeta_1$ the role as $\eta$, then 
$h=\partial f/\partial\eta=\partial f/\partial\zeta_1$.
If we let 
$$
w=f, \quad  z_j=\zeta_{j+1}, \ j=1,\ldots, n,
$$
then, since $\Theta=1$ in \eqref{svin2},    
Proposition~\ref{snabel2} provides the operators 
$$
1, \  \frac{\partial}{\partial\zeta_1}, \quad  \frac{\partial}{\partial\zeta_k}+\frac{\partial f}{\partial\zeta_k}
\frac{\partial}{\partial\zeta_1}, \    k=2, \ldots, n+1.
$$
They generate the same $\Ok_Z$-module as
$1$, $\partial/\partial\zeta_k,  \   k=1, \ldots, n+1$,
and by obvious invariance they thus generate $\N_X$ on $Z_{reg}$, so there is no need to consider
operators obtained by other choices of $\eta$.
Moreover, the coherent extension across $Z_{sing}$ is 
independent of the choice of global coordinates $\zeta$  in $\U$.
\end{ex}


\subsection{Global pointwise norm on $X$}
In Example~\ref{kollekt}  the  extension of  $\N_X$ across $Z_{sing}$ is invariant.
We do not know whether this is true in general. In any case we can define a global
pointwise norm in the following way:
Each point $x\in Z_{sing}$ has a \nbh $\U_x$ where we have a coherent extension by
Theorem~\ref{oxet} and in $\U_x$ we thus have a pointwise norm
$|\cdot|_{X,x}$.
We can choose a locally finite open covering $\{ \U_{x_j}\}$ of $X$, and
a partition of unity  $\chi_j$ partition of unity subordinate to the covering and define
the global norm
$$
|\cdot|_X^2=\sum_j  \chi_j |\cdot |^2_{X, x_j}.
$$

\section{Pseudomeromorphic currents}\label{pm}

Let $Y$ be a reduced analytic space. The $\Ok_Y$-sheaf $\PM_Y$ of pseudomeromorphic currents on
$Y$ was introduced in \cite{AW2, AS}. Roughly speaking, it 
consists of currents that locally are finite sums of direct images (under possibly nonproper mappings)
of products of simple principal value currents and $\dbar$ of such currents. See, e.g., \cite{AW3}
for a precise  definition and basic properties.   The sheaf  $\PM_Y$ is closed under $\dbar$ and under
multiplication by smooth forms.  If $\tau$ is \pm  in an open subset $\U\subset Y$
and $W\subset\U$ is a subvariety then there is a well-defined current $\1_{\U\setminus W}\tau$ in $\U$
obtained by extending the natural restriction of $\tau$ to $\U\setminus W$ in the trivial way. With the notation
in Section~\ref{sep},  $\1_{\U\setminus W}\tau=\lim_\epsilon\chi(|h|/\epsilon)\tau$ if $h$ is a tuple of
holomorphic functions with common zero set $W$.  Thus $\1_W\tau:=\tau-\1_{\U\setminus W}\tau$ is a \pm 
current with support
on $W$.   If $W'\subset\U$ is another subvariety, then
\begin{equation}\label{kamrat}
\1_{W'}\1_W\tau=\1_{W'\cap W}\tau.
\end{equation}
We can rephrase the standard extension property, cf.~Section~\ref{sep}:
If $\tau$ has support on a subvariety $Z$ of pure dimension, then $\tau$ has the SEP with respect to
$Z$ if  for each open subset $\U\subset Y$ and subvariety $W\subset\U\cap Z$ with positive
codimension in $Z$,   $\1_W\tau=0$.

An important property is the dimension principle:  If  $\tau$ in $\PM_Y$ has bidegree $(*,p)$ and support on
a variety of codimension larger than $p$, then $\tau$ must vanish.

\smallskip

Recall that a current is semi-meromorphic if it is a smooth form times a meromorphic function,
considered as a principal value current.
We say that a current $\alpha$ in $\U\subset\C^N$ 
is {\it almost semi-meromorphic}, cf.~\cite{AW3}, 
if there is a modification $\pi\colon \tilde\U\to \U$ and a semi-meromorphic
current $\tilde\alpha$ in $\tilde\U$  such that $\alpha=\pi_*\tilde\alpha$. 
Notice that an almost semi-meromorphic $\alpha$ is smooth outside
an analytic set $W$ of positive codimension in $\U$.

\begin{ex}
Coleff-Herrera currents in $\U\subset\C^N$ are pseudomeromorphic.  
Almost semi-meromorphic currents are pseudomeromorphic and has the SEP
on $\U$. 
\end{ex}

In general one cannot multiply \pm currents. However, 
assume that  $\tau$ is \pm and $\alpha$ is 
almost semi-meromorphic in $\U$ and let $W$ be the analytic set
where $\alpha$ is not smooth. There is a unique \pm current $T$  
in $\U$  that coincides with the natural product $\alpha\w\mu$
in $\U\setminus W$ and such that $\1_W T=0$.
For simplicity we denote this current by $\alpha\w\mu$.  If $\alpha'$ is another
almost semi-meromorphic current in $\U$, then 
the expression $\alpha'\w\alpha\w\tau$ means $\alpha'\w(\alpha\w\tau)$.
The equality, \cite[Propostition~4.12]{AW3},  that 
\begin{equation}\label{commute}
\alpha'\w\alpha\w\tau=\alpha\w\alpha'\w\tau.
\end{equation}
always holds. However, in general it is {\it not} true that $\alpha'\w\alpha\w\tau=(\alpha'\w\alpha)\w\tau$.

\begin{ex} 
Let $f$ be a holomorphic function with non-empty zero set, let  $\alpha=1/f$, $\alpha'=f$, 
and let $\tau=\dbar/1/f)$ (or any nonzero current with support 
contained in $Z(f)$). 
 Then both sides of \eqref{commute} vanishes, but 
$\alpha' \alpha=1$ and so $(\alpha' \alpha)\tau=\tau$.
\end{ex}

Assume that $\tau$ is pseudomeromorphic,  $\alpha$ is almost semi-meromorphic,
$\xi$ is smooth, and  $V$ is any subvariety.   Then we have 
\begin{equation}\label{commute2}
\1_V \alpha\w \tau=\alpha\w \1_V \tau.
 \end{equation}
In particular:  If $\tau$ has support on and the SEP with respect to $Z$,  then also $\alpha\w\tau$ has 
(support on and) the SEP with respect to $Z$.

\section{Uniform limits of holomorphic functions}\label{proofB}

Let $\J$ be the ideal sheaf in $\U\subset\C^N$ so that $\Ok_X=\Ok_\U/ \J$ as before.
Let $E_k$ be Hermitian vector bundles, $E_0$ a trivial line bundle, with morphisms
$f_k\colon E_k\to E_{k-1}$, such that  
\begin{equation}\label{upp}
0\to \Ok(E_N)\stackrel{f_N}{\to} \cdots \stackrel{f_2}{\to}\Ok(E_1)\stackrel{f_1}{\to}\Ok(E_0)\to \Ok(E_0)/\J\to 0
\end{equation}
is  a free resolution of $\Ok_\Omega/\J$ in $\U$. 
In \cite{AW1} was introduced a 
residue current $R=R_p+\cdots +R_N$ with support on $Z$, where
$R_k$ have bidegree $(0,k)$ and take values in $\Hom(E_0,E_k)\simeq E_k$,
such that $f_{k+1} R_{k+1}-\dbar R_k=0$ for each $k$, which can be written more compactly as
$$
(f-\dbar)R=0
$$
if $f:=f_1+\cdots + f_N$.  
The current $R$ has the additional property that a holomorphic function $\Phi$ in $\U$ belongs to $\J$ if and only
if the current $\Phi R=0$.  In particular, $\phi R$ is a well-defined current for $\phi$ in $\Ok_X$.  
The assumption that $X$ has pure dimension implies that $R$ has the SEP with respect to $Z$,
see \cite[Section~3]{AS} or \cite[Section~6]{AL} for a proof.

Recall that $\phi$ is a meromorphic function on $X$ if 
$\phi=g/h$ where $h$ is not nilpotent, i.e., a representative of $h$
does not vanish identically on $Z$, 
and $g/h=g'/h'$ if $gh'-g'h=0$ in $\Ok_X$. 
Because of the SEP the product $\phi  R$ is a well-defined \pm current in $\U$ if $\phi$ is  meromorphic on $X$.
The following criterion for holomorphicity was proved in \cite{Astrong}.

\begin{thm}\label{strong} 
Assume that $i\colon X\to \U$ has pure dimension and $R$ is a current as above. 
If $\phi$ is meromorphic on $X$, then it is holomorphic if and only if
\begin{equation}\label{krit}
(f-\dbar)(\phi R)=0.
\end{equation}
\end{thm}

\begin{remark}
In \cite{AL} is introduced a notion of currents on $X$ so that 
\eqref{krit} can
be formulated as an intrinsic condition on $X$.
\end{remark}

To give the idea for the general case
let us first sketch a proof of Theorem~\ref{thmB}, relying on Theorem~\ref{strong},  in case $X$ is reduced. 

\begin{proof}[Proof of Theorem~\ref{thmB} in case $X$ is reduced] 
The statement is elementary on $X_{reg}$; moreover it is clear that $\phi_j\to \phi$ where $\phi$ is
weakly holomorphic and thus meromorphic on $X$. 

There is a  (unique) almost semi-meromorphic 
current $\omega$ on $X$ of bidegree $(n,*)$ such that  $i_* \omega=R\w dz$,
where  $(z_1,\ldots,z_N)$ are coordinates in $\U$, see \cite[Proposition~3.3]{AS}.
In particular, $\omega$ has the SEP on $X$.  
Let $\pi\colon X'\to X$ be a smooth modification so that $\omega=\pi_*\omega'$, where
$\omega'$ is semi-meromorphic. 
 
%
Since $\pi^*\phi_j\to \pi^*\phi$ uniformly and $X'$ is smooth, indeed $\pi^*\phi_j\to \pi^*\phi$
in $\E_{X'}$.  
Therefore  $\pi^*\phi_j \omega'\to \pi^*\phi \omega'$. 
Since $\phi_j$ are smooth,  $\pi_*(\pi^*\phi_j \omega')=\phi_j \omega$.
Combining we find that
\begin{equation}\label{toto}
\phi_j\omega\to \pi_*(\pi^*\phi \omega').
\end{equation}
Since $\omega'$ has the SEP, so have $\pi^*\phi\omega'$ and hence $\pi_*(\pi^*\phi \omega')$.
Moreover,   
\begin{equation}\label{orkan}
\pi_*(\pi^*\phi \omega')=\phi \omega
\end{equation}
on the open subset of $X$ where $\phi$ is holomorphic, thus at least on $X_{reg}$.  
Since both sides of \eqref{orkan} have the SEP and 
coincide outside a set of positive codimension we conclude that they indeed coincide on $X$. 
Therefore \eqref{toto} means that 
$\phi_j\omega\to \phi\omega
$
and applying $i_*$ we get 
$
\phi_j R\to \phi R
$
and  hence
$
(f-\dbar)(\phi_j R)\to (f-\dbar)(\phi R).
$
It now follows from Theorem~\ref{strong} that 
$\phi$ is indeed holomorphic. 
 \end{proof}

For the rest of this section we will discuss the proof of the non-reduced case
when $Z$ is smooth.  

\begin{lma}\label{chcase}
Theorem~\ref{thmB} is true when $Z$ is smooth and $\Ok_X$ is Cohen-Macaulay.
\end{lma}

\begin{proof}
Consider an open  set $\U\subset X$ where we have
coordinates $(z,w)$, a basis $w^{\alpha_\ell}$ and  $\Ok_X$ is Cohen-Macaulay. Then we have unique
representatives in $\U$, cf.~Section~\ref{kolsyra}, 
$$
\hat\phi_j=\sum_\ell \hat\phi_{jl}(z)\otimes w^{\alpha_\ell}
$$
of $\phi_j$ in Theorem~\ref{thmB}.  By the hypothesis and Theorem~\ref{thmA}~(iii) it  follows that 
$\hat\phi_{j\ell}$ is a Cauchy sequence for each fixed $\ell$, and hence we have holomorphic
limits $\hat\phi_\ell=\lim_j\hat\phi_{j\ell}$ for each $\ell$. Let us define the function  
$$
\hat\phi:=\sum_\ell \hat\phi_\ell(z)\otimes  w^{\alpha_\ell}
$$
in $\U$ and let $\phi$ be its pullback to $\Ok_X$. 
Since the convergence holds for all derivatives of $\hat\phi_{j\ell}$ as well,
it follows from \eqref{atlas4} that  $|\phi_j-\phi |_X\to 0$.  
\end{proof}

The non-Cohen-Macaulay case is trickier.
Let us first look at an example. 

\begin{ex} \label{opium3}
Consider the space $X$ in Section~\ref{kokong}.
If $\phi_j$ is  a sequence as in Theorem~\ref{thmB}, it follows from Lemma~\ref{chcase} 
that $\phi_j$ has a  holomorphic limit $\phi$ in $X\setminus \{0\}$.  
Let $\Lk$ be the Noetherian operator in Section~\ref{opium}
and recall  that $\Lk\phi$ is a well-defined function on $Z$. 
By the hypothesis in Theorem~\ref{thmB},  $\Lk\phi_j$ is a Cauchy sequence on  $Z$ and since
$\Lk\phi_j\to \Lk\phi$ in $Z\setminus\{0\}$ we conclude that $\Lk\phi_j\to \Lk\phi$ 
uniformly in $Z$. Since $\Lk\phi_j(0)=0$ therefore  $\Lk\phi(0)=0$, and thus
$\phi$ is $\Ok_X$-holomorphic in $X$, cf.~Lemma~\ref{opium2}.  
It follows that $|\phi_j-\phi|_X\to 0$ on $X$.
\end{ex}

We cannot see how the argument in Example~\ref{opium3} can be extended directly, so we have
to go back to the origins of our $\Lk_j$ which are the currents.

\begin{proof}[Proof of Theorem~\ref{thmB} when $Z$ is smooth]
The statement is local so we can fix a point $x\in X$ and an embedding 
$i\colon X\to \U$ as usual.  
After possibly shrinking $\U$ we may assume that we have a 
a Hermitian free resolution \eqref{upp} and the associated residue current $R$
in $\U$.  
We will use, \cite[Lemma~6.2]{AL}:

\begin{prop}\label{goliat}
There is a trivial vector bundle $F\to \U$ and an $F$-valued Coleff-Herrera current $\mu$
such that its entries generate $\Homs(\Ok_X,\CH_\U^Z)$, and
an almost semi-meromorphic current $\alpha=\alpha_0+\cdots +\alpha_n$, where
$\alpha_k$ have bidegree $(0,k)$ and take values in $\Hom(F,E_{p+k})$,
such that 
$$
R\w dz=\alpha\mu, \quad R_{p+k}\w dz=\alpha_k\mu, \ k=0,1,\ldots, n.
$$   
Moreover, $\alpha$ is smooth where $\Ok_X$ is Cohen-Macaulay.
\end{prop}

 Let $W$ be the subset of $Z$ where $\Ok_X$ is not Cohen-Macaulay.  It is well-known that
$W$ has at least codimension $2$ in $Z$.

\begin{lma}\label{ponny}
If $\phi$ is holomorphic in $X\setminus W$, then $\phi$ has a meromorphic extension
to $X$.  
\end{lma}

It seems that this result should be well-known but we provide a proof since we could not find a reference.

\begin{proof}
We can assume that we have coordinates $(z,w)$ as usual in $\U$.  
Let $\mu$ be the tuple above and consider the representation \eqref{kolibri}.  Fix a $x'\in Z$ 
where $\Ok_X$ is Cohen-Macaulay and 
a monomial basis $1,\ldots, w^{\alpha_{\nu-1}}$ for $\Ok_X$ over $\Ok_Z$ in a \nbh  $\U'$ of $x'$.
We then have the 
$M\times \nu$-matrix $T$ in $\U'$ that for each holomorphic $\phi$ in $\Ok(X\cap\U')$ maps
the coefficients of its representative $\hat\phi$ in this monomial basis onto the
coefficients of the expansion \eqref{kolibri0} of $\phi\mu$, cf.~Section~\ref{polyp}.

Notice that the entries in $T$ are $\C$-linear combinations of the coefficients
of the representation \eqref{kolibri} of $\mu$ in $\U$.  Thus $T$ has a holomorphic extension to $Z$
(we may assume that $Z$ is connected).  
As pointed out in Section~\ref{polyp},   $T$ is pointwise injective in $Z\cap \U'$ and hence,
possible after reordering of the rows,  $T=(T'\ \ T'')^t$ where $T'$ is  a
$\nu\times\nu$-matrix  that is invertible in $\U'$.   Thus $T'$ has a meromorphic inverse $S'$ in $Z$ and if
$S=(S' \ \ 0)$, then $ST=I$.

Since $\phi$ is holomorphic outside $W$, it defines a tuple $b$ in $\Ok_Z^M$ in $Z\setminus W$
via the representation \eqref{kolibri0} of $\phi\mu$.  Since $W$ has at least codimension $2$
the tuple $b$ extends to $Z$.  
Now 
$$
\tilde \Phi:=\sum_j (Sb)_j(z)w^{\alpha_j}
$$
is a meromorphic function in $\U$ that defines a meromorphic function $\tilde\phi$ on $X$, since 
$(Sb)_j(z)$ are meromorphic on $Z$.  Moreover, $\tilde\Phi=\hat\phi$ in $\U'$ and so 
$\tilde\phi$ coincides with $\phi$ in $X\cap \U'$. 
By uniqueness 
$\tilde\phi=\phi$ in $\U\setminus W$ and thus $\tilde\phi$ is the desired meromorphic extension. 
\end{proof}

 If $\phi_j$ is a Cauchy sequence in $|\cdot|_X$-norm and $Z$ is smooth
 thus $\phi_j\to \phi$ uniformly on compact subsets
 of $X\setminus W$ by Lemma~\ref{chcase} and $\phi$ has a meromorphic extension to $X$
by Lemma~\ref{ponny}.


\begin{lma}\label{partout}  
With this notation 
 $\phi_j R\to \phi R$.
\end{lma}

\begin{proof}  
Choose a multiindex $M$ such that $\I:=\langle w^{M+\1}\rangle \subset \J$. 
As before let  $a$ be an  $F$-valued holomorphic function in $\U$ such that
$\mu=a\hat\mu$, cf.~\eqref{grobian}.    
Recall that
\begin{equation}\label{mor}
|\phi |_X=|\phi a|_{X'},
\end{equation}
where $\Ok_{X'}=\Ok/\I$. Define the $F$-valued $\Ok_{X'}$-functions $\psi_j=a\phi_j$.
It follows from the hypothesis and \eqref{mor} that $\psi_j$ is a  Cauchy sequence
with respect to $|\cdot|_{X'}$. Since $\Ok_{X'}$ is Cohen-Macaulay it follows from the proof of 
Lemma~\ref{chcase} that there is $\psi$ in $\Ok_{X'}$ and  
representatives $\hat\psi_j$ and $\hat\psi$ in $\U$ such that $\hat\psi_j\to \hat\psi$
in $\E(\U)$. 
Let $\Phi_j$ be representatives of $\phi_j$ in $\U$.
By Proposition~\ref{goliat}  and \eqref{commute} we have 
\begin{equation}\label{prutt0}
\phi_j R=\Phi_j R=\Phi_j \alpha \mu = \Phi_j \alpha a\hat\mu=\alpha \Phi_j a  \hat\mu=
\alpha (\Phi_j a)  \hat\mu=\alpha\hat\psi_j\hat\mu,  
\end{equation}
where the fifth equality holds since both $\Phi_j$ and $a$ are holomorphic, and the last 
equality holds since both $\Phi_j a$ and $\hat\psi_j$ are representatives in $\U$
of the class $\psi_j$ in $\Ok_{X'}$.  Since $\hat\psi_j\to \hat\psi$ in $\E(\U)$, 
$\alpha \hat\psi_j  \hat\mu\to \alpha \hat\psi  \hat\mu=\alpha \psi   \hat\mu$.
By \eqref{prutt0} thus 
\begin{equation}\label{prutt1}
\phi_j R\to \alpha \psi   \hat\mu.
\end{equation}
Let $\Phi$ be a representative in $\U$ of $\phi$. Since $\Phi$, $\alpha$ and $a$  are almost semi-meromorphic in $\U$,
 by \eqref{commute}, 
\begin{equation}\label{prutt2}
\phi R=\Phi R=\Phi \alpha\mu=\alpha \Phi \mu=\alpha \Phi a\hat \mu.
\end{equation}
We claim that
\begin{equation}\label{prutt3}
\alpha\Phi a \hat\mu =\alpha(\Phi a) \hat\mu.
\end{equation}
In fact, both $\Phi$ and $\alpha$ are almost semi-meromorphic in $\U$ and smooth
in a \nbh of each point on $Z$ where $\Ok_X$ is Cohen-Macaulay, cf.~Lemma~\ref{chcase}
and Proposition~\ref{goliat}.
Therefore \eqref{prutt3} holds in $\U\setminus W$, where $W\subset Z$ has positive codimension
in $Z$.  Both sides of \eqref{prutt3} have the SEP with respect to
$Z$, see Section~\ref{pm},  so \eqref{prutt3} holds everywhere.  
The right hand side of \eqref{prutt3} is equal to $\alpha\psi\hat\mu$, and so  Lemma~\ref{partout}
follows from \eqref{prutt1}, \eqref{prutt2}, and \eqref{prutt3}.
\end{proof}

Since $\phi_j$ are holomorphic, we have by Theorem~\ref{strong} and  Lemma~\ref{partout} that
$
0=\nabla_f(\phi_jR)\to \nabla_f(\phi R),
$
and hence $\phi$ is holomorphic in view of Theorem~\ref{strong}.  Now take $\Lk$  in $\N_X$.
By the hypothesis and definition of $|\cdot|_X$, $\Lk\phi_j$ is a holomorphic Cauchy sequence so it converges to a holomorphic
limit $H$. On the other hand we know that $\Lk\phi_j\to \Lk\phi$ where $\Ok_X$ is Cohen-Macaulay.
Thus $\Lk\phi_j\to \Lk\phi$ uniformly.  We conclude that $|\phi_j-\phi|_X\to 0$ uniformly.
Thus Theorem~\ref{thmB} is proved
when $Z$ is smooth. 
\end{proof}

\section{Resolution of $X$}\label{reso}
Assume that our $X$ of pure dimension $n$ is embedded in the smooth manifold $Y$ 
of dimension $N$ as before, and let $Z$ denote
the underlying reduced space. There exists a 
modification $\pi\colon Y'\to Y$ that is a biholomorphism $Y'\setminus \pi^{-1}Z_{sing}\simeq Y\setminus Z_{sing}$
and such that the strict transform $Z'$ of $Z$ is smooth and the restriction 
of $\pi$ to $Z'$ is a modification of $Z$.  Such a $\pi$ is called a strong resolution. 
Let $\widetilde\J$ be the ideal sheaf on $Y'$ generated by pullbacks of generators of $\J$ and consider the
relative gap sheaf $\J'=\widetilde\J [\pi^{-1}Z_{sing}]$, which is coherent, cf.\ \cite[Theorem 2]{ST}.   In fact,
one obtains $\J'$ by extending $\widetilde\J$ so that 
one gets rid of all primary components corresponding to the exceptional
divisor, and also possible embedded primary ideals in $Z'\cap \pi^{-1}Z_{sing}$.  
Thus  $\J'$
is the smallest coherent sheaf of pure dimension $n$ that contains $\widetilde \J$ and such that
$\Ok_{Y'}/\J'$ has support on $Z'$.  We let $X'$ denote the analytic space with structure sheaf
$\Ok_{X'}= \Ok_{Y'}/\J'$.   Notice that we have the induced mapping
\begin{equation}\label{pellets}
p^*\colon \Ok_{X}\to \Ok_{X'}.
\end{equation}
In fact, if $\phi\in \J$, then $\pi^*\phi\in\tilde\J\subset\J'$ so that $p^*$ in \eqref{pellets} is well-defined.
We say that $p\colon X'\to X$ is a resolution of $X$.  Notice that $p^*$ extends to map meromorphic 
functions on $X$ to meromorphic functions on $X'$.


\begin{lma}\label{gamas2}
Assume that $\phi'$ is meromorphic on $X'$ and holomorphic on $X'\setminus V$. 
Then there is a unique meromorphic $\phi$ on $X$, holomorphic in $X\setminus Z_{sing}$,  such that
$\phi'=p^*\phi$. 
\end{lma}

\begin{proof} 
Since $\pi$ is proper it follows from Grauert's theorem
that  the direct image $\F=\pi_* (\Ok_{Y'}/\J')$ is coherent,  and clearly it coincides with
$\Ok_Y\J$ outside $Z_{sing}\subset Y$.  Moreover, it contains $\Ok_X=\Ok_Y/\J$ since 
$\pi_* \pi^{-1}\phi=\phi$ for $\phi$ in $\Ok_X$.  Thus $\F/\Ok_X$ has support on $Z_{sing}$.
Let $h$ be a function that vanishes on $Z_{sing}$ but not identically in $Z$. Then 
$h^\nu \F/\Ok_X=0$ if $\nu$ is large enough. If $\phi'$ is a section of
$\Ok_{X'}$, therefore $g:=h^\nu \pi_*\phi'$ is holomorphic. Thus $\phi: =g/h^\nu$ is meromorphic
and $\phi'=p^*\phi$.
\end{proof}

\begin{lma}\label{sork}
Let $\mu$ be a tuple of currents that generate the $\Ok_X$-module  
$\Homs(\Ok_X,\CH_Y^Z)$.   

\smallskip
\noindent (i):   There is a unique tuple $\mu'$ of $(N,N-n)$-currents in $Y'$ with support on $Z'$ and the SEP
such that $\pi_*\mu'=\mu$. 

\smallskip
\noindent
 (ii)  A holomorphic function $\Phi'$ defined in a \nbh in $Y'$ of
a point on $Z'$ is in $\J'$ if and only if $\Phi' \mu'=0$. 
\end{lma}


In view of (ii) thus $\phi' \mu'$ is well-defined for $\phi'$ in $\Ok_{X'}$.   It is not
necessarily true that $\mu'$ is $\dbar$-closed. 
Let $p_0=\pi|_{Z'}$ so that 
\begin{equation}\label{vdef}
V:=p_0^{-1}Z_{sing}=\pi^{-1}Z_{sing}\cap Z'.
\end{equation} 
Since $\pi$ is a biholomorphism outside $\pi^{-1}Z_{sing}$ it follows
however that $\dbar\mu=0$ there.  In the literature such a $\mu'$ is often said to be a Coleff-Herrera current with
poles at $V\subset Z'$.  If $h'$ is holomorphic and vanishes
to enough order on $V$ then $0=h'\dbar \mu'=\dbar(h'\mu')$, and hence
$h'\mu'$ is in $\Homs(\Ok_{X'}, \CH_{Y'}^{Z'})$.

\begin{proof}
Recall that $\mu$ is pseudomeromorhic, cf.~Section~\ref{pm}.
By \cite[Theorem~2.15]{AW3} there is a current $T$ in $Y'$ such that  $\pi_* T=\mu$.   
Since $\pi$ is a biholomorphism outside $\pi^{-1}Z_{sing}$ the current $T$ must be unique there,
in particular it must have support on $\pi^{-1} Z$, and the SEP on $Z'\setminus\pi^{-1}Z_{sing}$.
If $W\subset Z'$ has positive codimension, therefore $\1_{W\setminus V} T= 0$.
 Let $\mu'=\1_{Z'\setminus V}T$. It follows from \eqref{kamrat} that
 $ \1_W\mu'=\1_{(Z'\setminus V)\cap W}T= \1_{W\setminus V}T=0$.
Thus $\mu'$ has support on $Z'$ and the SEP. Moreover,
 $\pi_*\mu'=\mu$ outside $Z_{sing}$ and hence everywhere by the SEP.
%
%

Since $\J'$ has no embedded components, $\Phi'$ is in $\J'$ if and only if $\Phi'$ is in
$\J'$ on $Z'\setminus V'$. 
This in turn holds if and only if $\Phi=\pi_*\Phi'$ belongs to
$\J$ on $Z_{reg}$ which holds if and only if $\Phi\mu=0$ on $Z_{reg}$. However this holds
if and only if $\Phi'\mu'=0$ on $Z'\setminus V$ which by the SEP of $\mu'$ holds if and only if
$\Phi'\mu'=0$ on $Z'$.
\end{proof}

Let $R$ be a current in $Y$ with support on $Z$ and the SEP as in Section~\ref{proofB}. 
Recall, Proposition~\ref{goliat},  that there is an almost semi-meromorphic current $\alpha$ in $Y$ 
such that
$R=\alpha\mu$ where $\mu$ is a tuple of Coleff-Herrera currents that generate
$\Homs(\Ok_X,\CH_Y^Z)$.

\begin{lma}\label{godis}
There is an almost semi-meromorphic current $\alpha'$ in $Y'$ such that $R'=\alpha' \mu'$
and $\pi_* R'=R$.
\end{lma}

\begin{proof} 
There is 
a modification $\tau\colon V\to Y$ such that $\alpha=\tau_* \gamma$, where $\gamma$ is
semi-meromorphic. There is a modification $V'\to Y$ that factors over both $V$ and $Y'$. 
Thus we get an almost semi-meromorphic $\alpha'$ in $Y'$ such that $\pi_*\alpha'=\alpha$.  
It follows that $R'=\alpha'\mu'$ since this holds outside $V$, cf.~\eqref{vdef}, and both currents have the SEP,
cf.~\eqref{commute2}. 
 \end{proof}

\section{Proof of Theorem~\ref{thmB}}\label{proofb}

\begin{lma}  \label{pellets2}
Assume that $Z$ is smooth and that $\Lk$ is a holomorphic differential operator on $X$ that belongs to
$\N_X$ in $Z\setminus W$, where $W$ has positive codimension. If $Z(h)\supset W$, then
$h^r\Lk$ is in $\N_X$ for large enough $r$. 
\end{lma}

\begin{proof}  
Recall that  the sheaf $\N_X$  locally can be considered 
as a coherent submodule of $\Ok_Z^\nu$ for some large $\nu$.
If  $\Lk$ is not in $\N_X$, then $\M'=\langle \J,\Lk\rangle /\N_X$ is a coherent sheaf with support on $W$. 
By the Nullstellensatz  $h^r\M'=0$ for  large enough $r$.   Thus  $h^r\Lk\in\N_X$ for such $r$.
\end{proof}

In a \nbh $\U\subset\subset X$ of a given point $x\in Z_{sing}$ we have a coherent extension of $N_X$,
cf.~Theorem~\ref{oxet},  that we denote by
$\N_X$ as well.
It is enough to prove Theorem~\ref{thmB} in $\U$. 
Let $\Lk_1,\ldots,\Lk_m$ be generators for the $N_X$ in $\U$, 
and let  $\U'=\pi^{-1}\U$.

\begin{lma}\label{morot}
There are meromorphic differential operators $\Lk'_j$  in $X'\cap \U'$ with poles
on  $V=p_0^{-1}Z_{sing}$ such that 
$
\Lk'_j(p^*\phi)= p_0^*(\Lk_j\phi)
$
on $\U'\cap Z'\setminus V$.  
If $h$ is a holomorphic in $\U$, vanishes
to high order on $Z_{sing}$,  and $h'=\pi^*h$, then $h'\Lk_j$ are in $\N_{X'}$. 
\end{lma}

Since  $Z'$ is smooth we have the well-defined $\Ok_{X'}$-module $\N_{X'}$ of Noetherian operators
on $X'$.

\begin{proof}
Given a holomorphic differential operator $T$ on $\U$ there is a holomorphic differential operator $\widetilde T$
in $\U'=\pi^{-1}\U$ with values in a power  $N_{Y'/Y}^\nu$ of the relative canonical bundle, and a holomorphic section
$s$ of $N_{Y'/Y}$, vanishing on $\pi^{-1}Z_{sing}$, such that 
$
\pi^*(T\Phi)=s^{-\nu}\widetilde T(\pi^*\Phi).
$
See, e.g., the discussion preceding \cite[Corollary 4.26]{AW3}.  Thus $T'=s^{-\nu}\widetilde T$
is a meromorphic differential operator such that $\pi^*(T\Phi)=T'(\pi^*\Phi)$.

Let $L_j$ be differential operators in $\U$ that define $\Lk_j$ and
let $L_j'$ be meromorphic differential operators in $\pi^{-1}\U$ such that $\pi^*(L_j\Psi)=L_j'(\pi^*\Psi)$.
 If $\Phi'$ is in $\J'$ then $L'\Phi'=0$ on
$Z'\setminus V$ and by continuity also  on $Z'$.  Thus $L'_j$ are Noetherian with respect to $\J'$.
Now let  $\Lk_j'$ be the induced Noetherian operators on $X'$.  
%
Since $p$ is a biholomorphism of non-reduced spaces outside $V$ it follows that
$\Lk'_j$ must belong to $N_{X'}$ there.
 %
%
If $h$ vanishes enough on $Z_{sing}$ thus   
 $h' L_j'$ and hence  $h' \Lk_j'$ are holomorphic. 
After possibly modifying $h$ it follows from Lemma~\ref{pellets2} that $ h' \Lk_j'$ are in $\N_{X'}$.
\end{proof}

\begin{lma}\label{pellets3}
After possibly shrinking $\U\ni x$ there is a holomorphic function $H$ in $\U$, not vanishing identically
on $Z$,  such that 
\begin{equation}\label{pellets4}
|p^*(H\phi)(z')|_{X'}\le C| \phi(\pi(z')) |_X, \quad z'\in Z'\cap \U'. 
\end{equation}
\end{lma}

\begin{proof}
Let $\widehat\N_{X'}$ be the  $\Ok_{Z'}$-module generated by the $h'\Lk_j'$.
Then $\widehat\N_{X'} \subset \N_{X'}$  with equality outside $Z(h')$.
Therefore $\N_{X'}/\widehat\N_{X'}$ is annihilated  by
$H'=\pi^*H$ if $H$ is a high power of $h$.   That is,  if $\T$ is in $\N_{X'}$, then $ H' \T$
is in $\widehat\N_{X'}$ and thus $H'\T$ is an $\Ok_{X'}$-linear combination of the $h' \Lk_j'$.  

Fix a point $x'\in \pi^{-1}(x)\cap Z'$.
Let $\T_\ell$ be a set of generators for $\N_{X'}$ in a \nbh $\V$ of $x'$.   
For any $\phi$ we have, with $\phi'=\pi^*\phi$, and $z'\in \V$, 
\begin{multline*}
|H'(z')||\phi'(z')|_{X'}\sim \sum_\ell \big|(H'\T_\ell\phi')(z')\big|\lesssim  \sum_j  \big|(h'\Lk_j'\pi^*\phi')(z')\big|\le \\
\sum_j  \big|\pi^*(\Lk_j\phi)(z')\big|\sim |\phi (\pi(z'))|_X.
\end{multline*}
On the other hand, if $\nu$ is large enough,  
$
|\T_\ell ((H')^\nu\phi')|\lesssim |H'\T_\ell\phi' |
$
for each $\ell$ and hence
$
| (H')^\nu\phi'|_{X'}\lesssim |H'| |\phi'|_{X'}.
$
Denoting $H^\nu$ by $H$  thus \eqref{pellets4} holds for $z'\in \V$.  Since $\pi^{-1}(x)$ is compact,
\eqref{pellets4} holds for all $z'$ in an open \nbh of $\pi^{-1}(x)$. Hence the lemma follows. 
\end{proof}

Assume that $\phi_j$ is a sequence as in Theorem~\ref{thmB} and let $\phi'_j=p^*\phi_j$.
It follows from Lemma~\ref{pellets3}, and Theorem~\ref{thmB} in  case that $Z$ is smooth, 
see, Section~\ref{proofB}, that
there is a holomorphic function $\xi'$ on $X'\cap\U'$ such that $H'\phi'_j\to \xi'$ 
uniformly in the $|\cdot|_{X'}$-norm.  Notice that $\xi'/H'$ is meromorphic on $X'\cap\U'$.

\begin{lma}\label{potatis}
With the notation above,
$
\phi'_j R'\to (\xi' /H') R'
$
in $\U'$.
\end{lma}

\begin{proof}
Since  $\dbar\mu'$ has support on $V$, 
$\dbar (g'\mu')=0$ for a suitable $g'=\pi^* g$ not vanishing identically on $Z'$.  From Lemma~\ref{sork}
we conclude that $g'\mu'$ is a tuple in $\Homs(\Ok_{X'}, \CH_{\U'}^{Z'})$. 
If $(z,w)$ are coordinates in $\V\subset\U'$ such that $Z'\cap\V=\{w=0\}$, then  
$g'\mu'=a dz\w \hat\mu$, cf.~\eqref{grobian}, 
for a suitable holomorphic tuple $a$ in $\V$.  Using \eqref{commute} and Lemma~\ref{godis} 
we can now prove Lemma~\ref{potatis} in $\V$
in the same way as Lemma~\ref{partout}. By compactness Lemma~\ref{potatis} holds in $\U'$.
\end{proof}

By Lemma~\ref{gamas2} there is a meromorphic  $\xi$ on $X\cap \U$ such that $\xi'=p^*\xi$.
Define the meromorphic function $\phi=\xi/H$ on $X\cap\U$.  Clearly $p^*\phi=\xi'/H'$ so that
\begin{equation}\label{slut}
\pi_*((\xi'/H') R')=\phi R
\end{equation}
outside $Z(H)\cap Z$.  However, both sides of \eqref{slut} has the SEP with respect to $Z$
so the equality holds in $\U$.  Since $\pi_*(\phi'_j R')=\pi_*(p^*\phi_j R')=\phi_j R$ we 
conclude from Lemma~\ref{potatis} that $\phi_j R\to \phi R$. 
In view of Theorem~\ref{strong}  now Theorem~\ref{thmB} follows as  in the smooth case
in Section~\ref{proofB}.

\end{document}